\documentclass[11 pt]{amsart}
\usepackage{ graphicx, mathrsfs}
\usepackage{amssymb,amsmath,amsthm}
\usepackage{mathrsfs,dsfont}

\theoremstyle{plain}
\newtheorem{theorem}{Theorem}

\newtheorem{proposition}{Proposition}[section]
\newtheorem{lemma}[proposition]{Lemma}
\newtheorem{defi}[proposition]{Definition}
\newtheorem{coro}[proposition]{Corollary}
\newtheorem*{remark}{Remark}
\newcommand{\RR}{\mathbb{R}}

\newcommand{\p}{\partial}

\newcommand{\eps}{\epsilon}

\begin{document}

\author{P. Germain, N. Masmoudi and J. Shatah}

\title{Global solutions for 2D quadratic Schr\"odinger equations}

\maketitle

\begin{abstract}
We prove global existence and scattering for a class of quadratic Schr\"odinger equations in dimension 2. The proof relies on the idea of space-time resonance.
\end{abstract}

\maketitle

\section{Introduction}

In the present article we examine   global existence  and asymptotic behavior of  solutions with small initial data for nonlinear Schr\"odinger equations with quadratic nonlinearities in dimension 2.
We believe that this particular model  is a good representative of a wider class of  \emph{weakly dispersive nonlinear equations}, i.e.  nonlinear dispersive equations where  the linear  decay due to dispersion is not strong enough a priori to overcome the nonlinear effects over large intervals of time: estimates relying only on the power of the nonlinearity, but not  on its structure, fail.   As we will argue later, the key concept in this setting becomes \emph{space-time resonances}.


\subsection{Known results}
Consider a nonlinear Schr\"odinger equation
$$
\partial_t u + i \Delta u = N_p(u) \qquad (t,x) \in \mathbb{R} \times \mathbb{R}^d,
$$
where $u$ is complex-valued, and $N_p$ a (nonlinear) function of $u$ which is homogeneous of order $p$.
We will review results concerning global existence and asymptotic behavior for small solutions. We refer to the textbooks by  Cazenave~\cite{Cazenave} and Tao~\cite{Tao} for a more general discussion.

The simplest problem  occurs if the decay given by the linear part is strong enough, or $p$ large enough to use  dispersive or Strichartz estimates to conclude that asymptotic completeness holds for small data, i.e. the wave operators are defined, and are one to one. For smaller $p$, more interesting effects appear, and the structure of the nonlinearity starts to play a role. Two values of $p$ are particularly important: the Strauss exponent $\displaystyle (\sqrt{d^2+12d+4}+d+2)/2d$ \cite{str}, and the short range exponent $\displaystyle 1 + 2/d$, whose values are displayed below for small dimensions.
\begin{center}
  \begin{tabular}{| c | c | c | }
    \hline 
    space dimension & short range exponent & Strauss exponent \\ \hline
    $1$ & $3$ & $(\sqrt{17}+ 3)/2$ \\ \hline
    $2$ & $2$ & $\sqrt{2}+1$ \\ \hline
    $3$ & $\frac{5}{3}$ & $2$ \\ \hline
  \end{tabular}
\end{center}

For $p$ larger than the Strauss exponent,  one expect   the existence of global solutions for small data, as well as some kind of asymptotic completeness.  For $p$ larger than the short range exponent  one expects the existence of wave operators, while for $p$ less than  the short range exponent  one expects  small solutions not to be asymptotically free.  Various global existence results for small solutions  will be illustrated below.

\subsubsection*{Wave operators for small data at $t=\infty$} 
Suppose first that $p$ lies above the short-range exponent, $p > 1 + \frac 2d$; an immediate computation shows that the solution becomes asymptotically free if it decays in $L^\infty$ at the rate prescribed by the linear part: $t^{-d/2}$.  In all known cases, wave operators can be constructed for small data for this range of $p$, but no general result seems available.  For the nonlinearity $N_p(u)= \pm |u|^{p-1}u$, see in particular Cazenave and Weissler~\cite{CW}, Ginibre, Ozawa and Velo~\cite{GOV} and Nakanishi~\cite{Nakanishi}. For small $p$ within this range, the spaces in which these wave operators exist involve weights or vector fields.

Consider now the case where $p$ lies below the short-range exponent $p \leq 1 + \frac 2d$. For the nonlinearity $N_p(u)=|u|^{p-1} u$, it was proved by Barab~\cite{Barab} that non trivial asymptotically free states cannot exist. Modified wave operators were subsequently constructed by Ozawa~\cite{Ozawa}, and Ginibre and Ozawa~\cite{GinibreOzawa} if $p= 1 + \frac 2d$. For the nonlinearity $|u|^2$, in dimension 2, it was proved by Shimomura~\cite{Shimomura} and Shimomura and Tsutsumi~\cite{ST} that non trivial asymptotically free states cannot exist either. However, for the nonlinearities $u^3$, $u\bar u^2$, $\bar u^3$ in dimension 1, and $u^2$, $\bar u^2$ in dimension 2, wave operators were constructed  by Moriyama, Tonegawa and Tsutsumi~\cite{MTT} and Shimomura and Tonegawa~\cite{ST2}; see also Hayashi, Naumkin, Shimomura and Tonegawa~\cite{HNST}. Finally, Gustafson, Nakanishi and Tsai~\cite{GNT1} proved, in dimensions 2 and 3, the existence of wave operators for the nonlinearity $(u + 2 \bar u + |u|^2)u$ arising from the Gross-Pitaevskii equation.

\subsubsection*{Global existence and asymptotic behavior for small data at $t=0$} For $p$ larger than the Strauss exponent  one can construct global solutions for small data using simply a fixed point theorem and dispersive estimates  \cite{str}\footnote{This can be seen as follows: let $- r(p)= \frac{d}{p}- \frac{d}{2} $ be the decay of $\|u(t)\|_p$ prescribed by the linear Schr\"odinger flow; then $p$ is larger than the Strauss exponent if and only if $p\,r(p+1)>1$. Thus one can easily get the global a priori estimate $\|u(t)\|_{p+1} \lesssim t^{-r(p+1)}$ for $u$ solving $\partial_t u + i\Delta u = N_p(u)$, and $u(t=0)=u_0$ small (for the sake of simplicity, we ignore the divergence of the integral for $s-t$ close to 0):
\begin{equation*}
\begin{split}
\left\|u(t)\right\|_{p+1} & \lesssim \|e^{it\Delta} u_0\|_{p+1} +  \int_0^t \left\|e^{i(t-s)\Delta} N_p(u(s))\right\|_{p+1} \,ds
\lesssim \|u_0\|_{\frac{p+1}{p}} t^{-r(p+1)} + \int_0^t (t-s)^{-r(p+1)} \left\|u(s)\right\|_{p+1}^p\,ds.
\end{split}
\end{equation*}}.
This holds regardless of the precise form of $N_p$, and furthermore, the solution scatters, ie it is asymptotically free in a certain sense.
For the short-range exponent $p = 1 + \frac{2}{d}$, Hayashi and Naumkin~\cite{HN1} showed that a modification must be added to the free solution to describe the behavior for large time. Apart from a result of Tsutsumi Yajima~\cite{TY}, who  prove scattering in the defocusing case, we are not aware of any result for the intermediary range (between the short-range and the Strauss exponent); one would however not expect scattering to hold in general.

For other nonlinearities, there are few known examples of global existence for small data below the Strauss exponent. In dimension 3 however, global existence and scattering are known for $u^2$ and $\bar u^2$: see Hayashi and Naumkin~\cite{HN2}, Kawahara~\cite{Kawahara} and Germain, Masmoudi and Shatah~\cite{GMS1}. For $|u|u$, this is also the case (Cazenave and Weissler~\cite{CW}), but for $|u|^2$ only almost global existence is known (Ginibre and Hayashi~\cite{GH}).  For the Gross-Pitaevskii equation, Gustafson, Nakanishi and Tsai~\cite{GNT2} proved the existence of global solutions which scatter for large time.

It is interesting to notice that for the Schr\"odinger equation, there is to our knowledge no known example of a nonlinearity which yields blow up in finite time for small, smooth and localized data. Such a nonlinearity should, as we have seen, necessarily correspond to a power below the Strauss exponent. For the nonlinear wave equation, we know since John~\cite{John} and Schaeffer~\cite{Schaeffer} that blow up occurs for nonlinearities which have the homogeneity of the Strauss exponent.

\subsubsection*{Global existence if the nonlinearity involves derivatives of $u$}
As will become clear in this article, derivatives in a nonlinearity can play the role of a null form, thus making estimates easier as far as  resonances are concerned. However in the presence of derivatives in the nonlinearity  one needs to recover the derivative loss in the estimates, thus making them more complicated. To shorten the discussion, we focus here on recent developments corresponding to nonlinearities of low power.

In dimension 3, Hayashi and Naumkin~\cite{HN3} were able to prove global existence and scattering for small data and for any quadratic nonlinearity involving at least one derivative: $u \nabla u$, $\bar u \nabla u$, $\bar u \nabla \bar u$... In dimension 2, Cohn~\cite{Cohn} obtained the same result for a nonlinearity of the type $\nabla \bar u \nabla \bar u$ (his proof, relying on a normal form transform and the use of pseudo-product operators, is actually very similar to parts of the arguments of the present article).
Finally, the main result is due to Delort~\cite{Delort}, who proved global existence for a nonlinearity of the form $u \nabla u$ or $\bar u \nabla \bar u$. His method combines the vector fields method, a normal form transform, and microlocal analysis; it enables him to prove global existence, but not scattering. Our approach to the question of global existence is quite different from his. The Fourier analysis we develop is essentially a new point of view on the vector field and normal form methods.

\subsection{The notion of space-time resonance} The concept of space-time resonance is a natural generalization of resonance for ODEs. If one considers a linear dispersive equation on $\RR^n$
\[
\p_t u = i L \Big(\frac 1 i \partial \Big) u
\]
then the quadratic time resonances can be found by considering plane wave solutions
\(
u = e^{i (L (\xi) t + \xi \cdot x )} .
\)
In this case time resonance for $u^2$ corresponds to
\[
\mathscr{T} =
\{ (\xi_1, \xi_2 ) ;
L(\xi_1) + L(\xi_2)
= L (\xi_1 + \xi_2) \}.
\]
However, time resonances tell only part of the story
for dispersive equations when one considers spatially localized solutions.
Specifically, if one considers two solutions
$u_1$ and $u_2$ with data localized in
space around the origin
and in frequency around $\xi_1$ and $\xi_2$, respectively,
then the solutions $u_1$ and $u_2$ at large time~$t$
will be spatially localized around
\(
(-\partial L (\xi_1) t)
\)
and $(- \partial L (\xi_2) t) $.
Thus quadratic spatial resonance is defined as the set $(\xi_1, \xi_2 ) \in \mathscr{S}$ where
\[
\mathscr{S}= \{ (\xi_1, \xi_2) ;
\partial L(\xi_1) = \partial L (\xi_2) \}.
\]
We define quadratic space-time resonance as
\[
\mathscr{R} = \mathscr{T} \cap \mathscr{S}.
\]
The idea is that only frequencies in $\mathscr{R}$ play a significant role in the long-term behavior of nonlinear dispersive equations. Indeed, the interaction between frequencies which are not time resonant is harmless, whereas frequencies which are not space resonant cannot interact since they have disjoint support - to be precise, this last point is valid only if the nonlinearity is local.

We believe that space-time resonances provide a key to understand the global behavior of nonlinear dispersive equations, for small data at least. We have been using this notion, along with its natural analytical framework, to study three-dimensional nonlinear Schr\"odinger equations~\cite{GMS1}, and more recently, water waves~\cite{GMS2}.

\medskip

What heuristic understanding of quadratic nonlinear Schr\"odinger equations does the notion of space-time resonance give? The three possible polynomial nonlinearities are $u^2$, $\bar u^2$, and $|u|^2$. An elementary computation (see Section~\ref{resonances}) shows that for the two first, $\mathscr{R} $ is reduced to a point, whereas it is a $d$-dimensional subspace for the third one.  This explains why, in dimension $d=3$, global existence can be proved relatively easily for $u^2$ and $\bar u^2$, whereas for $|u|^2$ only almost global existence is known.
In dimension $2$, the decay given by the linear Schr\"odinger equation is only $\frac{1}{t}$; in other words, quadratic nonlinearities are short-range, making global existence results very delicate. Actually, the only known results hold for nonlinearities of the type $u \nabla u$ or $\bar u \nabla \bar u$; more precisely: the nonlinearities for which global existence holds exclude interactions between $u$ and $\bar u$, and involve derivatives. Why are derivatives in the nonlinearity helpful as far as global existence is concerned? This can be understood by going back to the space-time resonant set, which is equal to  the zero frequencies of the interacting waves; these zero frequencies are canceled by derivatives.

The above considerations lead us to the choice of a quadratic nonlinearity $Q(u,u)$ in the theorem below. For low frequencies, which is where resonances occur, a derivative is needed to play the role of a null form, thus $Q(u,u)$ will look like $u\nabla u$. Taking $Q$ of the same form for high frequencies would lead to a problem distinct of resonances, which is our primary focus, namely: how to use the smoothing effect of the equation to ``recover'' derivatives. Since we want to avoid this technical complication, we simply define $Q(u,u)$ to be a standard product for high frequencies.

\subsection{Main result}  Consider the following equation on $u$, a complex-valued function of $(t,x) \in \mathbb{R}\times \mathbb{R}^2$,
\begin{equation*}
\tag{NLS}
\left\{
\begin{array}{l}
\partial_t u + i \Delta u = \alpha Q(u,u) + \beta Q(\bar u, \bar u) \\
u_{|t=2} = u _2 \stackrel{def}{=} e^{-2i\Delta} u_\ast,
\end{array}
\right.
\end{equation*}
where $\alpha$, $\beta$ are complex numbers and $Q$ is defined by
$$
\widehat{Q(f,g)}(\xi) = \int q(\xi,\eta) \widehat{f}(\eta) \widehat{g}(\xi-\eta) d\eta,
$$
$\widehat{\cdot}$ denoting the Fourier transform,  and where the symbol $q$ is smooth,  linear for $|(\xi,\eta)| \leq 1$,   and equal to $1$ for $|(\xi,\eta)| \geq 2$.  
Thus $Q$ is like a derivative for low frequencies, and the identity for high frequencies.

\begin{remark}
The fact that the data are given at time $2$ does not have a deep meaning: it is simply more convenient when performing estimates, since the $L^\infty$ decay of $\frac{1}{t}$ given by the linear part of the equation is not integrable at $0$.
\end{remark}

Before stating the theorem, let us introduce the profile $f$ given by $f(t) \overset{def}{=} e^{it\Delta} u(t)$.
\begin{theorem}
\label{schrod2d}
There exists $\epsilon>0$ such that if $u_\ast$ satisfies
$$
\left\| \langle x \rangle^2 u_\ast \right\|_2 \leq \epsilon,
$$
then there exists a global solution $u$ of $(NLS)$ such that
$$
\|\langle x \rangle f\|_2 \lesssim \epsilon \;\;,\;\;\|x^2 f \|_2 \lesssim  \epsilon+ \epsilon^2 t \;\;\mbox{and}\;\;\|e^{it\Delta} f\|_\infty \lesssim \frac{\epsilon}{t}.
$$
Furthermore, this solution scatters i.e. there exists $f_\infty \in L^2$ such that
$$
\left\| f(t) - f_\infty \right\|_2 \longrightarrow 0 \;\;\;\;\;\;\mbox{as $t\rightarrow \infty$}.
$$
\end{theorem}

\begin{remark}
Using the tools developed in this article, more general nonlinear Schr\"odinger equations can be treated, we give below a few examples.
\begin{enumerate}
\item The conclusion of the theorem still holds if any cubic terms of polynomial type are added, that is for the following equation
$$
\partial_t u + i \Delta u = \alpha Q(u,u) + \beta Q(\bar u, \bar u) + \gamma uuu + \delta uu\bar u + \epsilon u \bar u \bar u + \zeta \bar u \bar u \bar u
$$
(notice that it is not trivial to obtain the $L^\infty$ decay proved in the theorem even if the nonlinearity consists \emph{only} of cubic terms).
\item The theorem can be extended in a straightforward way to systems for which no quadratic or cubic space-time resonances occur.
\item Finally, it is possible to handle more general pseudo-products than $Q$.
It should also be possible to extend our result to the case where $Q(u,v)=\alpha u \nabla v + \beta \bar u \nabla \bar v$ by analyzing high frequencies more carefully than we have done.   Finally, we remark that the fact that $q$ is linear for low frequencies simplifies some manipulations in the following, but is not essential.
\end{enumerate}
\end{remark}

\subsection{Plan of the proof}

The article is structured as follows:
\smallskip

\noindent$\bullet$  In Section~\ref{resonances}, we analyze the resonant structure of the different terms of the equation, and perform a normal form transform on a certain part of the nonlinearity. This yields two terms, $g$ and $h = h_1 + h_2 +h_3$, which have different behaviors, and will satisfy different estimates, stated in section~\ref{outline}: (\ref{estimg}) for $g$ and (\ref{estimh}) for $h$. The proof of these estimates, performed in sections~\ref{sectiong} to~\ref{sectionh3} will give Theorem~\ref{schrod2d}.

\noindent$\bullet$   In Section~\ref{linear} we recall or establish  basic linear harmonic analysis results .

\noindent$\bullet$ In Section~\ref{multilinear} we turn to basic multilinear harmonic analysis, specifically  pseudo-product operators.

\noindent$\bullet$ In Section~\ref{sectiong}, the estimates~(\ref{estimg}) are established for $g$.

\noindent$\bullet$  In Section~\ref{esth1}, the estimates~(\ref{estimh}) are established for $h_1$.

\noindent$\bullet$ In Sections~\ref{h2hh} and~\ref{h2fg}, the estimates~(\ref{estimh}) are established for  $h_2$.

\noindent$\bullet$  In Section~\ref{sectionh3}, the estimates~(\ref{estimh}) are established for $h_3$.

\noindent$\bullet$Finally, in the appendix~\ref{appendix}, we prove boundedness of multilinear operators with flag singularities, a fundamental result of harmonic analysis that is needed in the proof.

\subsection{Notations}

\label{notations}  We denote by $C$ constants that may vary from one line to another, and
use the standard notation $A \lesssim B$ if there exists a constant $C$ such that $A \leq CB$, and $A \sim B$ if $B \lesssim A$ and $A \lesssim B$.
The Fourier transform of $f$ is denoted by  $\widehat{f}$ or $\mathcal{F}(f)$; the normalisation is the following
$$
\widehat{f}(\xi) = \frac{1}{2\pi} \int_{\mathbb{R}^2} e^{-ix\xi} {f}(x) \,dx.
$$
The Fourier multiplier with symbol $m$ is given by
$$
m(D) f \overset{def}{=} \mathcal{F}^{-1} m(\xi) \widehat{f} (\xi).
$$

\section{Computation of the resonances and first transformation of the equation}

\label{resonances}


Recall that $f$ denotes the profile of $u$
$
f(t,x) \overset{def}{=} e^{it\Delta} u(t,x) $ or  $\hat f (t,\xi)= e^{-i|\xi|^2t}\hat u(t,\xi)$.
Then
\begin{equation}
\label{derf}
\partial_t f (t,x) = e^{it\Delta} (\alpha Q(u,u) + \beta Q(\bar u, \bar u))
\end{equation}
thus
\begin{equation}
\label{duhamelversion}
\begin{split}
\hat{f}(t,\xi) = & \hat{u} _\ast(\xi) + \alpha \int _2^t \!\!\int e^{is\varphi_{++}(\xi,\eta)}  q(\xi,\eta) \hat{f}(s, \xi - \eta)  \hat{f}(s,\eta) d\eta  \, ds \\ 
& \;\;\;\; + \beta \int _2^t \!\!\int e^{is\varphi_{--}(\xi,\eta)} q(\xi,\eta) \hat{\bar f}(s, \xi - \eta)  \hat{\bar f}(s,\eta) d\eta  \, ds
\end{split}
\end{equation}
where 
$$
\varphi_{\pm \pm}  \overset{def}{=} -|\xi|^2 \pm |\eta|^2 \pm |\xi - \eta|^2.
$$

\subsection{Computation of the resonances}

The analysis that we will perform will rely on our understanding of resonances between two or three wave packets. In the present section, we describe the space, time, and space-time resonant sets; then we define cut-off functions, which split the $(\xi,\eta)$ or $(\xi,\eta,\sigma)$ plane into the different types of resonant sets.

\subsubsection{Quadratic resonances}

Due to our choice of nonlinearity, the only type of quadratic interactions occuring are two $+$ waves giving a $+$ wave or two $-$ waves giving a $+$ wave, or for short: ``$++$ gives $+$'' and ``$--$ gives $+$''. The corresponding phase functions are
$$
\varphi_{++}(\xi,\eta) = -|\xi|^2 + |\eta|^2 + |\xi - \eta|^2 \;\;\;\;\mbox{and}\;\;\;\; \varphi_{--}(\xi,\eta) = -|\xi|^2 - |\eta|^2 - |\xi - \eta|^2 .
$$
A simple computation gives that the space, time, and space-time resonant sets are: for $\varphi_{++}$
\begin{equation*}
\begin{split}
& \mathscr{S} _{++} = \{ \partial_\eta \varphi = 0 \} = \{ \xi = 2 \eta \} \\
& \mathscr{T} _{++} = \{ \varphi = 0 \} = \{ \eta \cdot (\xi-\eta) = 0 \} \\
& \mathscr{R} _{++} = \{ \partial_\eta \varphi = 0 \} \cap \{ \varphi = 0 \} = \{ \xi = \eta = 0 \} ,
\end{split}
\end{equation*}
and for $\varphi_{--}$
\begin{equation*}
\begin{split}
& \mathscr{S} _{--} = \{ \partial_\eta \varphi = 0 \} = \{ \xi = 2 \eta \} \\
& \mathscr{T} _{--} = \{ \xi = \eta = 0 \}  \\
& \mathscr{R} _{--} = \{ \xi = \eta = 0 \} .
\end{split}
\end{equation*}
In both cases the space-time resonant set is reduced to a point! This is to a large extent the key of the above theorem.

Further notice that as far as $\varphi_{--}$ is concerned, $\mathscr{T} _{--} = \mathscr{R} _{--}$; thus for this type of interaction, we shall not have to take space resonances into account for the analysis.

\bigskip

We take this opportunity to analyze the  $u \bar u = |u|^2$  interaction ( $+-$ gives $+$ ) and explain why  this interaction is  out of the scope of our theorem.  For  $+-$ gives $+$
one easily sees that
\begin{equation*}
\begin{split}
& \varphi_{-+}(\xi,\eta) = -|\xi|^2 - |\eta|^2 + |\xi - \eta|^2 = 2 \xi \cdot \eta \\
& \mathscr{S} _{-+} = \{ \xi = 0 \} \\
& \mathscr{T} _{-+} = \{ \xi \cdot \eta = 0 \} \\
& \mathscr{R} _{-+} = \{ \xi = 0 \} .
\end{split}
\end{equation*}
Thus, the space-time resonant set is too large; this explains why global existence should not be expected, or at least why our method does not apply.

\subsubsection{Cubic resonances}

\label{cubic}

All the possible cubic interactions,  namely ``$+++$ gives $+$'', ``$+--$ gives $+$'', ``$-++$ gives $+$''``- - - gives $+$'', occur for $\mathrm{(NLS)}$ as will become clear in the next section. They correspond respectively to the phase functions
\begin{equation}
\begin{split}
& \varphi_{+++} = - |\xi|^2 + |\xi - \eta|^2 + |\eta - \sigma|^2 + |\sigma|^2 \\
& \varphi_{+--} = - |\xi|^2 + |\xi - \eta|^2 - |\eta - \sigma|^2 - |\sigma|^2 \\
& \varphi_{-++} = - |\xi|^2 - |\xi - \eta|^2 + |\eta - \sigma|^2 + |\sigma|^2 \\
& \varphi_{---} = - |\xi|^2 - |\xi - \eta|^2 - |\eta - \sigma|^2 - |\sigma|^2 \\
\end{split}
\end{equation}
A small computation shows that the space-time resonant sets are: 
\begin{equation*}
\begin{split}
& \mathscr{S} _{+++} = \{ \partial_{\eta \sigma} \varphi = 0 \} = \{ \xi = 3 \sigma = \frac{3}{2} \eta \} \\
& \mathscr{T} _{+++} = \{ \xi^2 = (\xi-\eta)^2 + (\eta-\sigma)^2 + \sigma^2 \} \\
& \mathscr{R} _{+++} =  \{ \xi = \eta = 0 \} ,\\
&\\
& \mathscr{S} _{+--} = \{ \partial_{\eta \sigma} \varphi = 0 \} = \{ \xi = \sigma = \frac{1}{2} \eta \} \\
& \mathscr{T} _{+--} = \{\eta^2 =  \xi^2 + (\eta-\sigma)^2 + \sigma^2 \} \\
& \mathscr{R} _{+--} =  \{ \xi = \eta = 0 \} ,\\
&\\
& \mathscr{S} _{-++} = \{ \partial_{\eta \sigma} \varphi = 0 \} = \{ \xi = \sigma = \frac{1}{2} \eta \} \\
& \mathscr{T} _{-++} =  \{ \xi^2 + (\xi-\eta)^2 = (\eta-\sigma)^2 + \sigma^2 \} \\
& \mathscr{R} _{-++}  = \{ \xi = \sigma = \frac{1}{2} \eta \} ,\\
&\\
& \mathscr{S} _{---} = \{ \partial_{\eta \sigma} \varphi = 0 \} =\{ \xi = 3 \sigma = \frac{3}{2} \eta  \} \\
& \mathscr{T} _{---} = \{ \xi = \eta = \sigma = 0 \} \\
& \mathscr{R} _{---} = \{ \xi = \eta = \sigma = 0 \} .
\end{split}
\end{equation*}
Note that the  space-time resonant sets  $\mathscr{R} _{+++} = \mathscr{R} _{+--} = \mathscr{R} _{---} = \{ \xi = \eta = \sigma     = 0 \}$, which seems (and will be) favorable to obtain estimates.
The set $\mathscr{R} _{-++}= \{ \xi = \sigma = \frac{1}{2} \eta \}$, which looks very problematic is actually benign  since by  the following identity
\begin{equation}
\label{parrot}
\partial_\xi \varphi_{-++} = -2 \partial_\eta \varphi_{-++} - \partial_\sigma \varphi_{-++},
\end{equation}
it will generate ``null terms".  That is  when trying to establish the weighted $L^2$ estimate, one differentiates a certain trilinear expression in $\xi$, which corresponds to adding an $x$ weight in physical space. The worst term arises when the $\xi$ derivative hits an oscillating term with phase $\varphi_{-++}$, which introduces a factor of $s\partial_\xi \varphi_{-++}$. Due to the above identity, one can substitute to this factor $s(-2 \partial_\eta \varphi_{-++} - \partial_\sigma \varphi_{-++})$, which is harmless since an integration by parts in $\eta$ or $\sigma$ makes it disappear. See Section~\ref{sectionh3} for the details.

\subsubsection{Partition of the frequency space}

\label{partition}

The proof will rely on a decomposition of the multilinear expressions, which will be achieved by splitting the $(\xi,\eta)$, or $(\xi,\eta,\sigma)$ space; this manipulation will enable us to treat separately the different types of resonnances. 

Let us first explain the procedure in the case of quadratic interactions: consider either the $++$ or the $--$ case, and define 3 smooth functions $\chi^{\pm\pm,R}$, $\chi^{\pm\pm,S}$ and $\chi^{\pm\pm,T}$ of $(\xi,\eta)$ such that 
\begin{equation*}
\begin{split}
& 0 \leq \chi^{\pm\pm+,R}\,,\,\chi^{\pm\pm,S}\,,\,\chi^{\pm\pm,T} \leq 1 \;\;\mbox{and}\;\;\chi^{\pm\pm,R} + \chi^{\pm\pm,S} + \chi^{\pm\pm,T} = 1 \;\;\;\;\mbox{for any $(\xi,\eta)$} \\
& \chi^{\pm\pm,R} = 1 \;\;\mbox{on $B(0,1)$ and $0$ outside $B(0,2)$} \\ 
& \chi^{\pm\pm,T} \;\;\mbox{and}\;\; \chi^{\pm\pm,S} \mbox{are homogeneous of degree 0 outside $B(0,2)$.} \\
& \chi^{\pm\pm,T} = 0 \;\;\mbox{on a neighbourhood of $\mathscr{T} _{\pm\pm}$} \\
& \chi^{\pm\pm,S} =0 \;\;\mbox{on a neighbourhood of $\mathscr{S} _{\pm\pm}$.}
\end{split}
\end{equation*}
Of course, the splitting in the $--$ case is easier since time resonances are trivial then and one takes 
$$
\chi^{--,S} = 0 .
$$
The case of cubic resonances is handled similarly in the cases where the space-time resonant set is trivial, ie $+++$, $+--$ and $---$. This gives cut-off functions
$$
\chi^{\pm\pm\pm,R}\;\;\;,\;\;\;\chi^{\pm\pm\pm,S}\;\;\;\mbox{and}\;\;\;\chi^{\pm\pm\pm,T}.
$$
All the cut-off functions which have been defined will be dilated as time goes by, in the following way
$$
\chi^{\pm \pm,R,S,T}_t  \overset{def}{=} \chi^{\pm \pm , R,S,T}\left( \sqrt{t} \cdot \right)\;\;\;\;\;\mbox{and}\;\;\;\;\; \chi^{\pm \pm \pm,R,S,T}_t  \overset{def}{=} \chi^{\pm \pm \pm, R,S,T}\left( \sqrt{t} \cdot \right).
$$

\subsection{Normal form transform and decomposition of $f$}

\label{normalform}

Split the integral occuring in~(\ref{duhamelversion}) using the quadratic cutoff functions, and integrate by parts in $s$ the term with $\chi^T$, using the identity
$$
\frac{1}{i\varphi_{\pm \pm}(\xi,\eta)} \partial_s e^{is\varphi_{\pm \pm}(\xi,\eta)} = e^{is\varphi_{\pm \pm}(\xi,\eta)}.
$$
(this manipulation is nothing but a normal form transform). A small computation shows that the equation~(\ref{duhamelversion}) can then be rewritten as
\begin{equation}
\label{decompositionf}
\hat{f}(t,\xi) = \hat{u}_\ast(\xi) +  \hat{g}(t,\xi) + \hat{h}(t,\xi),
\end{equation}
with 
\begin{equation*}
 \hat{g}(t,\xi) = \left.\int\left( \alpha \frac{q(\xi,\eta)}{\varphi_{++}} \chi_s^{++,T}(\xi,\eta) e^{is\varphi_{++}}
 + \beta  \frac{q(\xi,\eta)}{\varphi_{--}} \chi_s^{--,T}(\xi,\eta) e^{is\varphi_{--}}\right)
  \hat{f}(s, \xi - \eta) \hat{f}(s,\eta) d\eta \right]_2^t 
\end{equation*}
and  all the remaining terms are denoted by 
$
\widehat{h}(t,\xi) = \widehat{h}_1(\xi) + \widehat{h}_2(\xi) + \widehat{h}_3(\xi)
$
where
\begin{equation*}
\label{dolphin}
\begin{split}
 \widehat{h}_1(\xi) \overset{def}{=}  &\alpha \int _2^t \!\!\int \chi_s^{++,R}(\xi,\eta)q(\xi,\eta) e^{is\varphi_{++}}  \widehat{f}(s, \xi - \eta)   \widehat{f}(s,\eta) d\eta   ds \\
& +\beta \int _2^t \!\!\int \chi_s^{--,R}(\xi,\eta)q(\xi,\eta) e^{is\varphi_{--}}  \widehat{\bar f}(s, \xi - \eta)   \widehat{\bar f}(s,\eta) d\eta   ds \\
&  - \alpha \int_2^t \int \partial_s \chi_s^{++,T}(\xi,\eta) \frac{q(\xi,\eta)}{i \varphi_{++}} e^{is\varphi_{++}} \widehat{f}(s, \xi - \eta)  \widehat{f}(s,\eta) d\eta   ds \\
& - \beta \int_2^t \int \partial_s \chi_s^{--,T}(\xi,\eta) \frac{q(\xi,\eta)}{i \varphi_{--}} e^{is\varphi_{--}} \widehat{\bar f}(s, \xi - \eta)  \widehat{\bar f}(s,\eta) d\eta   ds \\[2 em]
 \widehat{h}_2(\xi) \overset{def}{=} & \alpha \int_2^t \!\!\int \chi_s^{++,S}(\xi,\eta) e^{is\varphi_{++}} q(\xi,\eta) \hat{f}(s, \xi - \eta)  \widehat{f}(s,\eta)d\eta  ds \\[2 em]
 \widehat{h}_3(\xi) \overset{def}{=}  &-\alpha^2 \int_2^t \int \frac{\chi_s^{++,T} (\xi,\eta) q(\xi,\eta)  + \chi_s^{++,T} (\xi,\xi - \eta) q(\xi,\xi - \eta)}{i \varphi_{++}(\xi,\eta)} \\
&\qquad \qquad \times q(\eta,\sigma) e^{is\varphi_{+++}} \widehat{f}(s,\xi-\eta) \widehat{f}(s,\eta- \sigma) \widehat{f}(s,\sigma) d\eta \, d\sigma \,ds \\[1 em]
& -\alpha \beta \int_2^t \int \frac{\chi_s^{++,T} (\xi,\eta) q(\xi,\eta)  + \chi_s^{++,T} (\xi,\xi - \eta) q(\xi,\xi - \eta)}{i \varphi_{++}(\xi,\eta)} \\
&\qquad \qquad \times q(\eta,\sigma)e^{is\varphi_{+--}} \widehat{f}(s,\xi-\eta) \widehat{ \bar f}(s,\eta- \sigma) \widehat{\bar f}(s,\sigma) d\eta \, d\sigma \,ds \\[1 em]
& - \beta \bar \alpha \int_2^t \int \frac{\chi_s^{--,T} (\xi,\eta) q(\xi,\eta)  + \chi_s^{--,T} (\xi,\xi - \eta) q(\xi,\xi - \eta)}{i \varphi_{--}(\xi,\eta)} \\
&\qquad \qquad \times q(\eta,\sigma)e^{is\varphi_{---}} \widehat{\bar f}(s,\xi-\eta) \widehat{ \bar f}(s,\eta- \sigma) \widehat{\bar f}(s,\sigma) d\eta \, d\sigma \,ds \\[1 em]
& - |\beta |^2  \int_2^t \int \frac{\chi_s^{--,T} (\xi,\eta) q(\xi,\eta)  + \chi_s^{--,T} (\xi,\xi - \eta) q(\xi,\xi - \eta)}{i \varphi_{--}(\xi,\eta)} \\
&\qquad \qquad \times q(\eta,\sigma)e^{is\varphi_{-++}} \widehat{\bar f}(s,\xi-\eta) \widehat{ f}(s,\eta- \sigma) \widehat{ f}(s,\sigma) d\eta \, d\sigma \,ds \\
\end{split}
\end{equation*}
Thus $g$  consists of the boundary terms arising from integration by parts in  $s$, $h_1$ consists of terms that are strongly localized in frequency,  $h_2$ consist of quadratic terms, and $h_3$ consists of cubic terms.
The point  here is that $g$ and $h$ satisfy different types of estimates since  $g$ is less localized in space than $h$, but is  pointwise smaller.

\section{A priori estimates and outline of the proof}

\label{outline}

The proof of the theorem will consist in the following a priori estimates: for $g$,

\begin{equation}
\label{estimg}
\|g\|_2 \lesssim \frac{\epsilon^2}{\sqrt{t}} \quad  \;\;\;\;\|\langle x \rangle g\|_2 \lesssim \epsilon^2  \quad\;\;\;\;\|x^2 g \|_2 \lesssim \epsilon^2 t  \quad \;\;\;\;\|e^{it\Delta} g\|_\infty \lesssim \frac{\epsilon^2}{t} ,
\end{equation}
and for $h$,
\begin{equation}
\label{estimh}
\|\langle x \rangle h\|_2 \lesssim \epsilon^2  \quad \;\;\;\;\|x^2 h \|_2 \lesssim \epsilon^2 t^{5/8}  \quad \;\;\;\;\|e^{it\Delta} h\|_\infty \lesssim \frac{\epsilon^2}{t} .
\end{equation}
Since $f = u_\ast + g + h$, this implies
\begin{equation}
\label{estimf}
\|\langle x \rangle f\|_2 \lesssim \epsilon \quad   \;\;\|x^2 f \|_2 \lesssim \epsilon +  \epsilon^2 t \quad  \;\;\|e^{it\Delta} f\|_\infty \lesssim \frac{\epsilon}{t} .
\end{equation}
The above estimates will be established separately for $g$ and the three components of $h$, i.e.,  $h_1$, $h_2$ and $h_3$. Furthermore, it will be necessary to decompose $h_2$ further, by observing that $h_2$ can be seen as a bilinear operator and that
\begin{equation}
\begin{split}
h_2 & = h_2(f,f)   = h_2(u_\ast +g+h, u_\ast +g+h) \\
& = h_2(f,u_\ast) + h_2(u_\ast,g+h) + h_2(h,h) + h_2(g,h) + h_2(h,g) + h_2(g,g).
\end{split}
\end{equation}
Terms involving $u_\ast$ are the  simplest  to estimate and  we shall skip them.  Terms of the form  $h_2(h,h)$ and terms of the form $h_2(f,g)$ or $h_2(g,f)$ will be estimated in different ways.


{\it In order to simplify the notations, we will set in the following $\alpha$ and $\beta$ equal to 1, and we will denote indifferently $f$ for $f$ or its complex conjugate $\bar f$.}

\section{Linear harmonic analysis: basic results}

\label{linear}

The following are standard inequalities and notations that we include for the convenience of the reader.

\subsection{A Gagliardo-Nirenberg type inequality}  For Schr\"odinger equation the generator of the pseudo conformal transformation $J \overset{def}{=} x - 2it\partial$ plays the role of partial differentiation.  Thus we have

\begin{lemma}
\label{gagnir}
The following inequality holds
$$
\left\| e^{-it\Delta} (x f) \right\|_4^2 \leq \left\| e^{-it\Delta} f \right\|_\infty \left\| e^{-it\Delta} (x^2 f) \right\|_2 
$$
\end{lemma}
\begin{proof}  The proof relies on the observation that $e^{-it\Delta}x = J e^{-it\Delta}$, 
with  $J = 2it e^{-i\frac{x^2}{4t}} \partial e^{i\frac{x^2}{4t}}$.  Thus we get
\begin{equation*}
\begin{split}
\|e^{-it\Delta}xf\|_4 & = \|J e^{-it\Delta} f \|_4^2   = 4 t^2 \| e^{-i\frac{x^2}{4t}} \partial e^{i\frac{x^2}{4t}} e^{-it\Delta} f \|_4^2 
  \lesssim t^2 \| e^{-it\Delta} f \|_\infty \| \Delta e^{i\frac{x^2}{4t}} e^{-it\Delta} f \|_2\\
&   \lesssim \| e^{-it\Delta} f \|_\infty \| J^2 e^{-it\Delta} f \|_2    \lesssim \| e^{-it\Delta} f \|_\infty \| e^{-it\Delta} x^2 f \|_2,
\end{split}
\end{equation*}
where we used the standard  Gagliardo-Nirenberg inequality for the first inequality.
\end{proof}

\subsection{Littlewood-Paley theory}

\label{LPT}

Consider $\theta$ a function supported in the annulus $\mathcal{C}(0,\frac{3}{4},\frac{8}{3})$ such that
$$
\mbox{for $\xi \neq 0$,}\;\;\;\;\sum_{j \in \mathbb{Z}} \theta \left( \frac{\xi}{2^j} \right) = 1 .
$$
Define first
$$
\Theta(\xi) \overset{def}{=} \sum_{j <0} \theta \left( \frac{\xi}{2^j} \right)
$$
and then the Fourier multipliers
$$
P_j \overset{def}{=} \theta \left( \frac{D}{2^j} \right) \;\;\;\;\; P_{<j} = \Theta \left( \frac{D}{2^j} \right)\;\;\;\;\; P_{\leq j} = P_j + P_{<j} .
$$
This gives a homogeneous and an inhomogeneous decomposition of the identity (for instance, in $L^2$)
$$
\sum_{j \in \mathbb{Z}} P_j = \operatorname{Id} \;\;\;\;\mbox{and}\;\;\;\;\;P_{<0} + \sum_{j \geq 0} P_j = \operatorname{Id}.
$$
All these operators are bounded on $L^p$ spaces:
$$
\mbox{if $1<p<\infty$,}\;\;\;\; \|P_j f \|_p \lesssim \|f\|_p \;\;\;\;,\;\;\;\; \|P_{<j} f \|_p \lesssim \|f\|_p.
$$
Also recall Bernstein's lemma: if $1\leq q\leq p \leq \infty$,
\begin{equation}
\label{lemmadeltaj}
\|P_j f \|_p \leq 2^{2j\left( \frac{1}{q}-\frac{1}{p} \right)} \left\| P_j f \right\|_q\;\;\;\;\;\;\mbox{and}\;\;\;\;\; \left\| P_{<j} f \right\|_p \leq 2^{2j\left( \frac{1}{q}-\frac{1}{p} \right)} \left\| P_{<j} f \right\|_q .
\end{equation}

Finally, we will need the Littlewood-Paley square and maximal function estimates

\begin{theorem}
\label{LP}
(i) If $f = \sum f_j$, with $\operatorname{Supp}(f_j) \subset \mathcal{C}(0,c2^{-j},C2^{-j})$ (the latter denoting the annulus of center $0$, inner radius $c2^{-j}$, outer radius $C2^{-j}$), and $1 < p < \infty$,
$$
\left\| \sum_j f_j \right\|_p \lesssim \left\| \left[ \sum_j f_j^2 \right]^{1/2} \right\|_p.
$$
Furthermore, denoting $\displaystyle Sf \overset{def}{=} \left[ \sum_j (P_j f)^2 \right]^{1/2}$, $\displaystyle \|Sf\|_p \sim \| f \|_p .$

(ii) If $1 < p \leq \infty$, denoting $ \displaystyle Mf(x) \overset{def}{=} \sup_j \left| S_j f (x) \right|$, $\displaystyle \|Mf\|_p \lesssim \|f\|_p$.
\end{theorem}

\subsection{Fractional integration and dispersion} To some extent, the approach that we follow transforms the question `` how does the linear Schr\"odinger flow and resonances interact?'' into ``how can one combine fractional integration and the dispersive estimates for the Schr\"odinger group?''

The following lemma will thus be very useful. Define a smooth function $Z$  such that $Z(\xi)=|\xi|^{-1}$ 
for $|\xi|\geq 2$ and $Z(\xi)=1$ for $|\xi|\leq 1$. Then set for $\alpha \geq 0$
$$ 
\Lambda_t^{-\alpha} \overset{def}{=} \sqrt{t}^{\alpha} Z^\alpha  \left( \sqrt{t}|D| \right) ,
$$
thus $\Lambda_t^{-\alpha}$ is like fractional integration of order $\alpha$ for frequencies $\gtrsim \frac{1}{\sqrt{t}}$, and like $\sqrt{t}^{\alpha}$ for frequencies $\lesssim \frac{1}{\sqrt{t}}$.
\begin{lemma}
\label{boundlin}
\begin{itemize}
(i) If $\alpha \geq 0$, and either $1 \leq p,q \leq \infty$, and $0 \leq \frac{1}{q} - \frac{1}{p} < \frac{\alpha}{2}$, or $1 \leq p,q < \infty$ and $0 \leq \frac{1}{q} - \frac{1}{p} = \frac{\alpha}{2}$ there holds
$$
\left\| \Lambda_t^{-\alpha} f \right\|_p  \lesssim t^{\frac{\alpha}{2} + \frac{1}{p} - \frac{1}{q}} \|f \|_q  .
$$
(ii) If $1\leq p \leq 2$, there holds
$$
\left\| e^{it\Delta} f \right\|_{p'} \lesssim \frac{1}{t^{\frac{d}{p} - \frac{d}{2}}} \|f\|_p .
$$
(iii) If $1 \leq p \leq 2$, and $2^j t^2 \geq 1$
$$
\displaystyle \left\|P_{j} e^{it\Delta} f \right\|_p \lesssim \left(2^{2j} t \right)^{\frac{2}{p}-1} \|f\|_p .
$$
(iv) If $1 \leq q \leq 2 \leq p \leq \infty$, $\alpha \geq 0$, $1 \leq p,q < \infty$, and $0 \leq \frac{1}{q} - \frac{1}{p} \leq \frac{\alpha}{2}$, there holds
$$
\left\| \Lambda_t^{-\alpha} e^{it\Delta} f \right\|_{p} \lesssim t^{\frac{\alpha}{2} + \frac{1}{p} - \frac{1}{q}} \| f \|_q .
$$
\end{itemize}
\end{lemma}

\begin{proof} The points $(i)$ and $(ii)$ are standard. In order to prove $(iii)$, observe that it
follows from interpolation between the $L^2$ estimate, which is clear, and the $L^1$ estimate, which reads
$$
\mbox{if $2^j t^2 \geq 1$,}\;\;\;\;\;\;\;\;\left\| P_j e^{it \Delta} f \right\|_1 \lesssim 2^{2j} t  \|f\|_1 .
$$
By scaling, it suffices to prove this estimate if $t = 1$ and $j \geq 0$. This is done as follows
\begin{equation*}
\begin{split}
\left\| P_{j} e^{i\Delta} \right\|_{L^1 \rightarrow L^1} & \leq \left\| \mathcal{F}^{-1} \theta \left( \frac{\xi}{2^j} \right) e^{-i \xi^2} \right\|_1 
\lesssim \left\| \mathcal{F}^{-1} \theta \left( \frac{\xi}{2^j} \right) e^{-i|\xi|^2} \right\|_2^{1/2} 
 \left\| x^2  \mathcal{F}^{-1} \theta \left( \frac{\xi}{2^j} \right) e^{-i|\xi|^2} \right\|_2^{1/2}  \\
&  \lesssim  \left\|\theta \left( \frac{\xi}{2^j} \right) \right\|_2^{1/2} \left\|  \partial_\xi^2   \left[ \theta \left( \frac{\xi}{2^j} \right) e^{-i|\xi|^{2}} \right] \right\|_2^{1/2}  \lesssim 2^{2j} . 
\end{split} 
\end{equation*}


As for $(iv)$, it follows from $(i)$,  $(ii)$, and $\left\|  e^{it\Delta} \right\|_{L^2 \rightarrow L^2}=1$  
\begin{equation}
\left\| \Lambda_t^{-\alpha} e^{it\Delta} \right\|_{L^q \rightarrow L^p}  = \left\| \Lambda_t^{-\alpha \frac{\frac{1}{q}-\frac{1}{2}}{\frac{1}{q}-\frac{1}{p}}} e^{it\Delta} \Lambda_t^{-\alpha \frac{\frac{1}{2}-\frac{1}{p}}{\frac{1}{q}-\frac{1}{p}}}\right\|_{L^q \rightarrow L^p} \
 \leq \left\| \Lambda_t^{-\alpha \frac{\frac{1}{q}-\frac{1}{2}}{\frac{1}{q}-\frac{1}{p}}}\right\|_{L^q \rightarrow L^2}  \left\| \Lambda_t^{-\alpha \frac{\frac{1}{2}-\frac{1}{p}}{\frac{1}{q}-\frac{1}{p}}}\right\|_{L^2 \rightarrow L^p} .
\end{equation}
\end{proof}

\section{Multilinear harmonic analysis: pseudo-product operators}

\label{multilinear}


We only define bi and tri-linear pseudo-product operators, since these are the only cases that will be of interest in the following. These operators are defined by a symbol $m$ through
$$
T_{m(\xi,\eta)} (f_1,f_2) = {\mathcal{F}}^{-1} \int m(\xi,\eta) \hat{f}_1(\eta) \hat{f}_2(\xi-\eta) d\eta. 
$$
in the bilinear case and
$$
T_{m(\xi,\eta,\sigma)}(f_1,f_2,f_3) = {\mathcal{F}}^{-1} \int m(\xi,\eta,\sigma) \hat{f}_1(\sigma) \hat{f}_2(\eta-\sigma) \hat{f}_3(\xi-\eta) \,d\eta \,d\sigma
$$
in the trilinear case.

\subsection{Bounds for standard pseudo-product operators}

The fundamental theorem of Coifman and Meyer states, under a natural condition, that these operators have the same boundedness properties as the ones given by H\"older's inequality for the standard product.
\newtheorem*{nthcm}{Theorem C-M}
\begin{nthcm}[Coifman-Meyer]
\label{CoifmanMeyer}
Suppose that $m$ satisfies
\begin{equation}
\label{butterfly}
\|m\|_{CM} = \sup_{\xi,|\alpha_1| + \dots + |\alpha_n| \leq N} \left(|\xi_1| + \dots + |\xi_n| \right)^{|\alpha_1| + \dots + |\alpha_n|} |\partial_{\xi_1}^{\alpha_1} \dots \partial_{\xi_n}^{\alpha_n} m (\xi_1,\dots,\xi_n)|
\end{equation}
where $N$ is a sufficiently large number. Then the operator
$$
T_m : L^p \times L^q  \rightarrow  L^r
$$
is bounded for 
$
\frac{1}{r} = \frac{1}{p} + \frac{1}{q} , \quad 1<p,q \leq \infty\quad  \mbox{and} \quad 0<r<\infty   .
$
Furthermore, the bound is less than a multiple of $\|m\|_{CM}$.
\end{nthcm}

\begin{remark}
\label{CMCM}
1) For condition~(\ref{butterfly}) to hold, it suffices for $m$ to be homogeneous of degree $0$, and of class ${\mathcal{C}}^\infty$ on a $(\xi,\eta)$ sphere.  2)  If $m(\xi,\eta)$ is a Coifman-Meyer multiplier, so is $m_t(\xi,\eta) = m(t\xi,t\eta)$, for $t$ a real number. 
Furthermore, the bounds~(\ref{butterfly}) are independent of $t$, and consequently  so are the norms of $T_{m_t}$ as an operator from $L^p \times L^q$ to $L^r$, for $(p,q,r)$ satisfying the hypotheses of the Theorem.
\end{remark}

We now define a class of symbols which will be of constant use for us, due to the decomposition introduced in~(\ref{decompositionf}).

\begin{defi} 
1) We say that a symbol $\mu$ has homogeneous bounds of order $k$ (for a specified range of $(\xi,\eta)$ if it satisfies the estimates
$$
\left| \partial_{\xi_1}^{\alpha_1}\dots \partial_{\xi_n}^{\alpha_n} \mu(\xi_1,\dots,\xi_n) \right| \lesssim \left(|\xi_1| + \dots + |\xi_n| \right)^{k - |\alpha|}
$$
for sufficiently many multi-indices $\alpha$.

2) We denote $M^{k,k'}$ for a symbol smooth except at 0, such that
\begin{itemize}
 \item For $|(\xi_1,\dots,\xi_n)| \leq 2$, it has homogeneous bounds of order $k$.
\item For $|(\xi_1,\dots,\xi_n)| \geq 2$, it has homogeneous bounds of order $k'$.
\end{itemize}

3) We denote $m^{k,k'}_t$ for a symbol smooth except at 0, such that\footnote{Notice that this convention is similar to the ones for constants $C$: in the following, $m^{k,k'}_t$ stands for different symbols, as long as they are of the above type.}
\begin{itemize}
 \item For $|(\xi_1,\dots,\xi_n)| \leq 2$, $\displaystyle m^{k,k'}_t(\xi_1,\dots,\xi_n) = t^{-\frac{k}{2}}\mu (\sqrt{t}(\xi_1,\dots,\xi_n))$, where $\mu = 0$ in a neighborhood of $(0,0)$, and $\mu$ has homogeneous bounds of order $k$ for any $(\xi_1,\dots,\xi_n)$.

\item For $|(\xi_1,\dots,\xi_n)| \geq 2$, it has homogeneous bounds of order $k'$ and is independent of $t$.
\end{itemize}
\end{defi}
Thus one should think of a symbol $M^{k,k'}$ as a symbol of the form
\begin{equation}
\begin{split}
&\mbox{for $|(\xi_1,\dots,\xi_n)| \lesssim 1$,} \;\;M^{k,k'}(\xi_1,\dots,\xi_n) \sim (\xi_1,\dots,\xi_n)^k \\
&\mbox{for $|(\xi_1,\dots,\xi_n)| \gtrsim 1$,} \;\;M^{k,k'}(\xi_1,\dots,\xi_n) \sim (\xi_1,\dots,\xi_n)^{k'},
\end{split}
\end{equation}
whereas a symbol $m^{k,k'}_t$ looks like
\begin{equation}
\begin{split}
& \mbox{for $|(\xi_1,\dots,\xi_n)| \lesssim \frac{1}{\sqrt{t}}$,} \;\;m_t^{k,k'}(\xi_1,\dots,\xi_n) =0 \\
&\mbox{for $\frac{1}{\sqrt{t}} \lesssim |(\xi_1,\dots,\xi_n)| \lesssim 1$,} \;\;m^{k,k'}_t(\xi_1,\dots,\xi_n) \sim (\xi_1,\dots,\xi_n)^k \\
&\mbox{for $|(\xi_1,\dots,\xi_n)| \gtrsim 1$,} \;\;m^{k,k'}_t(\xi_1,\dots,\xi_n) \sim (\xi_1,\dots,\xi_n)^{k'}.
\end{split}
\end{equation}
In particular
$$
q(\xi,\eta) = M^{1,0} (\xi,\eta)\;\;\;\;\;\;\partial_{\xi,\eta} \varphi_{\pm,\pm} = M^{1,1}(\xi,\eta)
\;\;\;\;\;\;\partial_{\xi,\eta,\sigma} \varphi_{\pm,\pm,\pm} = M^{1,1}(\xi,\eta,\sigma)
.
$$
We now state a few calculus rules for these symbols.

\begin{proposition}
\label{calcrules}
Multiplications between symbols and differentiations of symbols satisfy
\[
\begin{split}
& M^{k,k'} m_t^{l,l'} = m_t^{k+l,k'+l'}   \qquad \qquad
 m_t^{k,k'} m_t^{l,l'} = m_t^{k+l,k'+l'} \\
& \partial_{\xi,\eta} M^{k,k'} = M^{k-1,k'-1} \qquad\qquad
 \partial_{\xi,\eta} m_t^{k,k'} = m_t^{k-1,k'-1} \qquad\qquad
 \partial_t m_t^{k,k'} = \frac{1}{t} m_t^{k,k'} .
\end{split}
\]
\end{proposition}
\begin{proof} Only the last assertion is not obvious. It follows from the identity
$$
\partial_t t^{-k/2} \mu(\sqrt{t} (\xi_1,\dots,\xi_n)) = -\frac{k}{2} t^{-\frac{k}{2}-1} \mu( \sqrt{t} (\xi_1,\dots,\xi_n)) + \frac{1}{2} t^{-\frac{k+1}{2}} (\xi_1,\dots,\xi_n) \cdot \partial \mu( \sqrt{t} (\xi_1,\dots,\xi_n)) .
$$
\end{proof}

The following shows how to combine fractional integration and bilinear operators; this will be exploited in the following corollary to get actual estimates.

\begin{lemma}
\label{boundbilin} Given a symbol $m^{k,k'}_t$, if $ k' \leq K \leq k$, there exist ($t$-dependent) symbols $m_1 \dots m_n$ satisfying (uniformly in $t$) the Coifman-Meyer bounds~(\ref{butterfly}) such that
\begin{equation*}
T_{m^{k,k'}_t} (f_1,\dots,f_n) =  \sum_{i=1}^n T_{m_i}(f_1, \dots ,\Lambda_t^{K} f_i, \dots f_n ) .
\end{equation*}
\end{lemma}

\begin{proof} So as to make notations lighter, we only prove the Lemma in the bilinear case $n=2$.
Let $\chi_1, \chi_2$ be functions of $\xi$ and $\eta$, homogeneous of degree $0$ and $\mathcal{C}^\infty$ outside $(0,0)$, such that
\begin{equation}
\label{psi1psi2}
\begin{split}
& \chi_1(\xi,\eta) + \chi_2(\xi,\eta)  = 1   \quad   \mbox{for any $(\xi,\eta)$} \\
& \mbox{on $\operatorname{Supp} \chi_1$,}\;\;|\eta| \lesssim |\xi-\eta| \\
& \mbox{on $\operatorname{Supp} \chi_2$,}\;\;|\xi-\eta| \lesssim |\eta| .
\end{split}
\end{equation}
We decompose the symbol $m_t^{k,k'}$ as follows
$$
m^{k,k'}_t(\xi,\eta) = \chi_1 (\xi,\eta) m^{k,k'}_t(\xi,\eta) + \chi_2 (\xi,\eta) m^{k,k'}_t(\xi,\eta) \overset{def}{=} m_1(\xi,\eta) + m_2(\xi,\eta) .
$$
By symmetry, it suffices to treat the case $|\eta| \lesssim |\xi-\eta|$, which corresponds to the support of $\chi_1$, hence to $m_1$. Then it suffices to observe that the symbol
$$
\frac{m_1(\xi,\eta)}{(\frac{1}{t}+|\xi-\eta|^2)^{K/2}} = \chi_1(\xi,\eta) \frac{m_t^{k,k'}(\xi,\eta)}{(\frac{1}{t}+|\xi-\eta|^2)^{K/2}}
$$
satisfies the Coifman-Meyer bounds~(\ref{butterfly}) with constants which are independent of $t$. 
\end{proof}

It follows from the above lemma and the Coifman-Meyer theorem that if $\frac{1}{r} = \frac{1}{p_1} + \dots + \frac{1}{p_n}$,
$$
\left\| T_{m_t^{k,k'}}(f_1,\dots,f_n) \right\|_r \lesssim \sum_{i=1}^n \|f_1\|_{p_1} \dots \|\Lambda_t^K f_i\|_{p_i}\dots \|f_n\|_{p_n}.
$$
Using furthermore Lemma~\ref{boundlin} gives

\begin{coro}
\label{coro1}
Suppose that $\frac{1}{r} = \frac{1}{p_1} + \dots + \frac{1}{p_n}$ and that $k' \leq K \leq k$. Then for any number $L \leq 0$
$$
\left\| T_{m_t^{k,k'}}(f_1,\dots,f_n) \right\|_r \lesssim t^{-L/2} \sum_{i=1}^n \|f_1\|_{p_1} \dots \|\Lambda_t^{K-L} f_i\|_{p_i}\dots \|f_n\|_{p_n} .
$$
In particular, if $0\geq k \geq k'$ or $k \geq 0 \geq k'$,
$$
\left\| T_{m_t^{k,k'}}(f_1,\dots,f_n) \right\|_r \lesssim t^{-k^-/2} \|f_1\|_{p_1} \dots \|f_{n}\|_{p_n} ,
$$
where $k^-=\min(0,k)$.
\end{coro}

\subsection{Bounds for pseudo-product operators with flag singularities}

The trilinear operators that will occur in our investigations will exhibit the following kind of singularity.

\begin{defi} 
\label{fstd0}
The symbol $m$ is called of flag singularity type with degree $0$ if it can be written as
$$
m(\xi,\eta,\sigma) = m^{III}(\xi,\eta,\sigma) m^{II}_1(\eta,\xi) m^{II}_2(\eta,\sigma) 
$$
where
$$
\|m\|_{FS} \overset{def}{=} \|m^{III}\|_{CM} \|m^{II}_1\|_{CM} \|m^{II}_2\|_{CM} < \infty.
$$
(Notice that this is not the most general instance of a flag singularity, but it will be sufficient for our purposes).
\end{defi}

It is not a priori clear that the operator associated to such a flag-singularity symbol enjoys the same boundedness properties as a Coifman-Meyer operator. This is in sharp contrast with the bilinear situation, where a symbol with a flag singularity is easily analyzed.

Boundedness of pseudo-products with flag singularities is given by the following theorem. Note that an instance of a paraproduct with flag singularity has been analyzed in Muscalu~\cite{Muscalu}, who derived much more general estimates than the ones we are about to state ; however, the type of pseudo-products that we have to deal with does not fit into his framework, hence the need of the following theorem proved in Appendix~\ref{appendix}.

\begin{theorem}
\label{FS}
(i) Suppose $m$ is of flag singularity type with degree 0 (see the above definition). Then the operator
$$
T_m : L^p \times L^q \times L^r \rightarrow  L^s
$$
is bounded for 
$
\frac{1}{s} = \frac{1}{p} + \frac{1}{q} + \frac{1}{r} , \quad 1<p,q,r,s<\infty.
$
Furthermore, the bound is less than a multiple of $\|m\|_{FS}$.

(ii) If $m$ is zero for $|\eta,\sigma| >> |\xi|$, then the operators $T_m(P_0 \cdot , P_{<2}\cdot , \cdot )$ and $T_m(P_{<2} \cdot,P_0\cdot,\cdot)$ are bounded from $L^\infty \times L^\infty \times L^2$ to $L^2$.
\end{theorem}
\begin{coro}
\label{coro2}
Suppose that $j \geq j'$, $k\geq k'$, $l \geq l'$, and that $j',k',l' \leq 0$.  
If furthermore $p$, $q$, $r$, $s$ satisfy the hypotheses of Theorem~\ref{FS}, then 
$$
\left\| T_{m^{j,j'}m^{k,k'}m^{l,l'}}(f_1,f_2,f_3) \right\|_s \lesssim t^{-\frac{j^-+k^-+l^-}{2}} \|f_1\|_{p_1} \|f_2\|_{p} \|f_{q}\|_{r} .
$$
\end{coro}

\section{Estimates on $g$}

\label{sectiong}

Recall that 
\begin{equation*}
 \hat{g}(t,\xi) = \left.\int\left( \alpha \frac{q(\xi,\eta)}{\varphi_{++}} \chi_s^{++,T}(\xi,\eta) e^{is\varphi_{++}}
 + \beta  \frac{q(\xi,\eta)}{\varphi_{--}} \chi_s^{--,T}(\xi,\eta) e^{is\varphi_{--}}\right)
  \hat{f}(s, \xi - \eta) \hat{f}(s,\eta) d\eta \right]_2^t 
\end{equation*}
Since the part corresponding to $s=2$ is very easy to estimate,  in the following we write 
$$
\hat{g}(t,\xi) = \int m^{-1,-2}_t e^{it \varphi(\xi,\eta)} \hat{f}(\xi - \eta) \hat{f}(\eta) d\eta,
$$
where $\varphi$ is either $\varphi_{++}$ or $\varphi_{--}$.

\subsection{Control of $g$ in $L^2$}

It follows from the above formula, Corollary~\ref{coro1} and Lemma~\ref{boundlin} that
\begin{equation*}
\begin{split}
\|g\|_2 & = \|e^{it\Delta} T_{m^{-1,-2}_t} (e^{-it\Delta} f, e^{-it\Delta} f) \|_2  \lesssim \| \Lambda_t^{-1} e^{-it\Delta} f \|_4 \| e^{-it\Delta} f \|_4 \\
& \lesssim \| f \|_{4/3} \| e^{-it\Delta} f \|_4 
 \lesssim \| \langle x \rangle f \|_2 \| e^{-it\Delta} f \|_4 \lesssim \frac{\eps^2}{\sqrt{t}}.
\end{split}
\end{equation*}

\subsection{Control of $xg$ in $L^2$}

Applying $\partial_\xi$ to $\widehat{g}(\xi)$ 
$$
\partial_\xi \widehat{g}(\xi) = \partial_\xi \int m^{-1,-2}_t  e^{it \varphi(\xi,\eta)} \hat{f}(\xi - \eta) \hat{f}(\eta) d\eta
$$
yields, by Proposition~\ref{calcrules}, terms of the following types
\begin{subequations}
\begin{align}
&\label{xg-1} \int t m^{0,-1}_t e^{it\varphi} \hat{f}(\eta) \hat{f}(\xi-\eta) d\eta \\
&\label{xg-2} \int m^{-2,-3}_t e^{it\varphi} \hat{f}(\eta) \hat{f}(\xi-\eta) d\eta \\
&\label{xg-3} \int {m^{-1,-2}_t} e^{it\varphi} \hat{f}(\eta) \partial_\xi \hat{f}(\xi-\eta) d\eta.
\end{align}
\end{subequations}
By Corollary~\ref{coro1},
\begin{equation*}
\begin{split}
\|(\ref{xg-1})\|_2 & = t \|e^{it\Delta} T_{m^{0,-1}_t} (e^{-it\Delta} f, e^{-it\Delta} f) \|_2 
 \lesssim t \| f\|_{2} \|e^{-it\Delta} f\|_{\infty} \lesssim \eps^2 .\\
\|(\ref{xg-2})\|_2 & = \|e^{it\Delta} T_{m^{-2,-3}_t} (e^{-it\Delta} f, e^{-it\Delta} f) \|_2 
 \lesssim t \|f\|_2 \|e^{-it\Delta} f\|_{\infty} \lesssim \eps^2 .
\end{split}
\end{equation*}
The last term,~(\ref{xg-3}), is estimated in a very similar way, so we skip it.

\subsection{Control of $x^2 g$ in $L^2$}

Applying $\partial_\xi^2$ to $\widehat{g}(\xi)$
$$
\partial_\xi^2 \widehat{g}(\xi) = \partial_\xi^2 \int m^{-1,-2}_t  e^{it \varphi(\xi,\eta)} \hat{f}(\xi - \eta) \hat{f}(\eta) d\eta
$$
yields, by Proposition~\ref{calcrules}, terms of the type
\begin{subequations}
\begin{align}
&\label{xxg-1} \int t^2 m^{1,0}_t e^{it\varphi} \hat{f}(\eta) \hat{f}(\xi-\eta) d\eta \\
&\label{xxg-2} \int t m^{-1,-2}_t e^{it\varphi} \hat{f}(\eta) \hat{f}(\xi-\eta) d\eta \\
&\label{xxg-3} \int m^{-3,-4}_t e^{it\varphi} \hat{f}(\eta) \hat{f}(\xi-\eta) d\eta \\
&\label{xxg-4} \int t m^{0,-1}_t e^{it\varphi} \hat{f}(\eta) \partial_\xi \hat{f}(\xi-\eta) d\eta \\
&\label{xxg-5} \int m^{-2,-3}_t e^{it\varphi} \hat{f}(\eta) \partial_\xi \hat{f}(\xi-\eta) d\eta \\
&\label{xxg-6} \int m^{-1,-2}_t e^{it\varphi} \hat{f}(\eta) \partial_\xi^2 \hat{f}(\xi-\eta) d\eta.
\end{align}
\end{subequations}

The first term is the one which gives a growth of $t$: Corollary~\ref{coro1} gives
\begin{equation*}
\|(\ref{xxg-1})\|_2  = t^2 \|e^{it\Delta} T_{m^{1,0}_t} (e^{-it\Delta} f, e^{-it\Delta} f) \|_2 
 \lesssim t^2 \|f\|_2 \|e^{-it\Delta} f\|_\infty 
 \lesssim  \eps^2  t^2 \frac{1}{t} \lesssim \eps^2   t .
\end{equation*}
The other terms are lower order, and can be controlled with the help of Corollary~\ref{coro1}. For instance,
\begin{equation*}
\|(\ref{xxg-3})\|_2  = \|e^{it\Delta} T_{m^{-3,-4}_t} (e^{-it\Delta} f, e^{-it\Delta} f) \|_2 
\lesssim t^{3/2} \|e^{-it\Delta} f\|_\infty \|f\|_2 
\lesssim t^{3/2} \eps^2  \frac{1}{t} \lesssim  \eps^2 \sqrt{t}.
\end{equation*}
The estimates for terms~(\ref{xxg-2})~(\ref{xxg-4})~(\ref{xxg-5}) follow in a similar manner.
Still with the help of Corollary~\ref{coro1}, one obtains the bound for the term~(\ref{xxg-6}):
\begin{equation*}
\|(\ref{xxg-6})\|_2  = \|e^{it\Delta} T_{m^{-1,-2}_t} (e^{-it\Delta} f, e^{-it\Delta} x^2 f) \|_2 
 \lesssim t^{1/2} \|e^{-it\Delta} f\|_\infty  \|x^2 f\|_2 \lesssim  \eps^2 \sqrt{t}.
\end{equation*}

\subsection{Control of $e^{-it\Delta} g$ in $L^\infty$}

Notice first that
$$
e^{it|\xi|^2} \widehat{g}(\xi) = \int m^{-1,-2}_t e^{it|\eta|^2} \widehat{f}(\eta) e^{it|\xi - \eta|^2} \widehat{f}(\xi - \eta) \,d\eta.
$$
Write the above using a rudimentary paraproduct decomposition
$$
\sum_j 2^{-j} \inf (1,2^{-j}) T_{m^{0,0}} \left( P_{<j} e^{-it\Delta} f , P_j e^{-it\Delta} f \right) + \sum_j 2^{-j} \inf (1,2^{-j}) T_{m^{0,0}} \left( P_j e^{-it\Delta} f , P_{\leq j} e^{-it\Delta} f \right)
$$
(where, as usual, $m^{0,0}$ stands for different symbols all belonging to the class defined in Section~\ref{multilinear}).
We only show how to deal with the first summand, the second one can of course be treated in the same way.
To bound it in $L^\infty$, we use repetitively Bernstein's inequality~(\ref{lemmadeltaj}) as follows
\begin{equation*}
\begin{split}
& \left\|  \sum_{j\leq 0} 2^{-j} T_{m^{0,0}}\left( P_{<j} e^{-it\Delta} f , P_j e^{-it\Delta} f \right) + \sum_{j>0} 2^{-2j} T_{m^{0,0}} \left( P_{<j} e^{-it\Delta} f , P_j e^{-it\Delta} f \right) \right\|_\infty \\
&\;\;\;\;\;\;\; \lesssim \sum_{j\leq 0} 2^{-j} 2^{j} \left\| T_{m^{0,0}}\left( P_{<j} e^{-it\Delta} f , P_j e^{-it\Delta} f \right)\right\|_2 + \sum_{j>0} 2^{-2j} 2^{j} \left\| T_{m^{0,0}} \left( P_{<j} e^{-it\Delta} f, P_j e^{-it\Delta} f \right) \right\|_2 \\
&\;\;\;\;\;\;\; \lesssim \sum_{j\leq 0} \left\| P_{<j} e^{-it\Delta} f \right\|_\infty \left\| P_j e^{-it\Delta} f \right\|_2 + \sum_{j>0} 2^{-j} \left\| P_{<j} e^{-it\Delta} f \right\|_\infty  \left\| P_j e^{-it\Delta} f \right\|_2 \\
&\;\;\;\;\;\;\; \lesssim \sum_{j\leq 0} 2^{j/2} \left\| e^{-it\Delta} f \right\|_\infty \left\| P_j  f \right\|_{4/3} + \sum_{j>0} 2^{-j} \left\| e^{-it\Delta} f \right\|_\infty  \left\|  e^{-it\Delta} f \right\|_2 \\
&\;\;\;\;\;\;\; \lesssim \sum_{j\leq 0} 2^{j/2} \frac{1}{t} \eps  \|\langle x \rangle f\|_{2} + \sum_{j>0} 2^{-j} \frac{\eps^2}{t}  \lesssim \frac{\eps^2}{t}.
\end{split}
\end{equation*}

\section{Estimates on $h_1$}

\label{esth1}

We observe that $h_1$ terms can all be written as
$$
\int_2^t \int \frac{1}{\sqrt{s}} m(\sqrt{s}(\eta,\xi)) \widehat{f}(\eta) \widehat{f}(\xi-\eta) \,d\eta\,ds,
$$
where $m$ is a smooth function with compact support.

\subsection{Control of $h_1$ in $L^2$}


The estimate follows naturally by the theorem of Coifman-Meyer:
\begin{equation*}
\begin{split}
\|h_1\|_2 & = \left\| \int_2^t \!\!\int  \frac{1}{\sqrt{s}} m(\sqrt{s}(\eta,\xi)) \hat{f}(s,\eta) \hat{f}(s, \xi - \eta) \, d\eta \,  ds \right\|_2 \\
& = \left\| \int_2^t \!\!\int  \frac{1}{\sqrt{s}} m(\sqrt{s}(\eta,\xi)) e^{-is\eta^2} e^{  is\eta^2} \hat{f}(s,\eta)\hat{f}(s, \xi - \eta)\,  d\eta \,  ds \right\|_2 \\
& \leq \int_2^t \frac{1}{\sqrt{s}} \left\|  T_{m(\sqrt{s}(\xi,\eta)e^{-is\eta^2}} \left( e^{  -is\Delta} f , f \right) \right\|_2 ds \\
& \lesssim \int_2^t \frac{1}{\sqrt{s}}  \| e^{  -is\Delta} f \|_\infty \|f\|_2 \,ds 
 \lesssim \int_2^t \frac{ \eps  }{\sqrt{s}} \frac{ \eps  }{s}\,ds \lesssim  \eps^2  .
\end{split}
\end{equation*}

\subsection{Control of $x h_1$ in $L^2$}


Applying $\partial_\xi$ to $\widehat{h}_1(\xi)$, terms of the following types appear 
\begin{subequations}
\begin{align}
& \label{xh1-1} \int_2^t \!\! \int m \left( \sqrt{s}(\eta,\xi) \right)   \hat{f}(s,\eta) \hat{f}(s, \xi - \eta)d\eta   ds \\
& \label{xh1-2} \int_2^t \!\!  \int \frac{1}{\sqrt{s}}m \left( \sqrt{s}(\eta,\xi) \right)    \hat{f}(s,\eta) \partial_\xi \hat{f}(s, \xi - \eta)d\eta   ds ,
\end{align}
\end{subequations}
where $m$ stands for a smooth compactly supported function. Terms of type~(\ref{xh1-2}) can be estimated precisely as above.

In order to treat the other kind of terms, observe that Bernstein's inequality~(\ref{lemmadeltaj}) gives
\begin{equation}
\label{addition}
\mbox{if $1<p<2$, }\;\;\left\|P_{<-\frac{1}{2}\log_2(t)-C} f \right\|_2 \lesssim_p t^{\frac{1}{2}-\frac{1}{p}} \|f\|_p \lesssim t^{\frac{1}{2}-\frac{1}{p}} \|\langle x \rangle f\|_2 \lesssim t^{\frac{1}{2}-\frac{1}{p}} \epsilon .
\end{equation}
Therefore, picking some $p$ between $1$ and $2$, 
\begin{equation*}
\begin{split}
\|(\ref{xh1-1})\|_2 & = \left\| \int_2^t \!\! \int m \left( \sqrt{s}(\eta,\xi) \right) \hat{f}(s, \xi - \eta)  \hat{f}(s,\eta) d\eta   ds \right\|_2 \\
& = \left\| \int_2^t \!\! \int m \left( \sqrt{s}(\eta,\xi) \right) e^{  is\eta^2}e^{- is\eta^2}  \hat{f}(s,\eta) \hat{f}(s, \xi - \eta)d\eta   ds \right\|_2 \\
& \lesssim \int_2^t \left\||  T_{m(\sqrt{s}(\xi,\eta))e^{  is\eta^2}} \left( e^{-is\Delta} f , P_{<-\frac{1}{2}\log_2(s)-C}f  \right) \right\|_2 ds \\
& \lesssim \int_2^t \|e^{  -is\Delta} f\|_\infty \|P_{<-\frac{1}{2}\log_2(s)-C} f \|_2  ds 
 \lesssim \epsilon^2 \int_2^t s^{\frac{1}{2}-\frac{1}{p}} s^{-1} ds \lesssim \epsilon^2 .
\end{split}
\end{equation*}

\subsection{Control of $x^2 h_1$ in $L^2$}

Applying $\partial_\xi^2$ to $\widehat{h_1}$ yields terms of the type 
\begin{subequations}
\begin{align}
& \label{xxh1-1} \int_2^t \!\! \int \sqrt{s} m \left( \sqrt{s}(\eta,\xi) \right)  \hat{f}(s,\eta)\hat{f}(s, \xi - \eta)  d\eta   ds \\
& \label{xxh1-2} \int_2^t \!\! \int m \left( \sqrt{s}(\eta,\xi) \right) \hat{f}(s,\eta) \partial_\xi\hat{f}(s, \xi - \eta)   d\eta   ds \\
& \label{xxh1-3} \int_2^t \!\!  \int \frac{1}{\sqrt{s}}m \left( \sqrt{s}(\eta,\xi) \right) \hat{f}(s,\eta) \partial_\xi^2 \hat{f}(s, \xi - \eta)   d\eta   ds ,
\end{align}
\end{subequations}
where $m$ stands for a smooth compactly supported function. Terms of the type~(\ref{xxh1-2}) or~(\ref{xxh1-1}) can be treated as in the previous paragraphs ; as for the last type of terms, the estimate is straightforward
\begin{equation*}
\begin{split}
\|(\ref{xxh1-3})\|_2 & = \left\| \int_2^t \!\!  \int \frac{1}{\sqrt{s}}m \left( \sqrt{s}(\eta,\xi) \right)   \hat{f}(s,\eta) \partial_\xi^2 \hat{f}(s, \xi - \eta)   d\eta   ds \right\|_2 \\
 & \lesssim \int_2^t \frac{1}{\sqrt{s}}  \left\| e^{  -is\Delta} f  \right\|_\infty \left\|  x^2 f \right\|_2 ds 
  \lesssim \int_2^t   \eps^2  \frac{1}{\sqrt{s}} s \frac{1}{s} ds \lesssim  \eps^2  \sqrt{t} .
\end{split}
\end{equation*}

\subsection{Control of $e^{  -is\Delta} h_1$ in $L^\infty$}

In order to show $\left\| e^{  -is\Delta} h_1 \right\|_\infty \lesssim \frac{ \eps^2  }{{t}}$, it suffices to prove that $\|h_1\|_1 \lesssim   \eps^2  $. 
In general, the oscillating phases are a hindrance to obtaining $L^1$ estimates, but due to the shrinking support of $m(\sqrt{s} \cdot)$, oscillations do not occur here. Therefore, using~(\ref{addition}),
\begin{equation}
\begin{split}
\|h_1\|_{1} & = \left\| \int_2^t  \frac{1}{\sqrt{s}} T_{m(\sqrt{s}(\eta,\xi))} ( f , f)   ds \right\|_1 \\
& = \left\|  \int_2^t  \frac{1}{\sqrt{s}} T_{m(\sqrt{s}(\eta,\xi))} \left( P_{<-\frac{1}{2}\log_2(s) - C}  f , P_{<-\frac{1}{2}\log_2(s) - C}  f \right) \,ds \right\|_1 \\
& \lesssim \int_2^t \frac{1}{\sqrt{s}} \left\|  P_{<-\frac{1}{2}\log_2(s)-C} f\right\|_2 \left\|  P_{<-\frac{1}{2}\log_2(s)-C} f \right\|_2  ds  \\
& \lesssim \int_2^t \frac{1}{\sqrt{s}} \left( s^{-\frac{3}{8}} \| f \|_{7/8} \right)^2  ds 
 \lesssim \int_2^t  \eps^2 \frac{1}{\sqrt{s}} s^{-\frac{3}{8}} s^{-\frac{3}{8}} ds  \lesssim \eps^2 .
\end{split}
\end{equation}

\section{Estimates on $h_2(h,h)$}

\label{h2hh}

As explained in Section \ref{outline},   we shall in the present section derive estimates on $h_2(h,h)$ and $h_2(f,g)$ seperately.
Since $\partial_\eta \varphi_{++}$ does not vanish on the support of $\chi^S_s$, the idea in order to estimate $h_2$ will always be to integrate by parts in $\eta$ using the identity
$$
\frac{1}{i s (\partial_{\eta} \varphi_{++})^2} \partial_{\eta} \varphi_{++} \cdot \partial_\eta e^{is\varphi_{++}} = e^{is\varphi_{++}} .
$$
that we write symbolically
\begin{equation}
\label{ibp}
\frac{1}{s} M^{-1,-1} \partial_\eta e^{is\varphi_{++}} = e^{is\varphi_{++}} .
\end{equation}

\subsection{Control of $h_2(h,h)$ in $L^2$}

\label{h2hhl2}

As far as the $L^2$ estimate is concerned, it is not necessary to distinguish between $h_2(h,h)$ and $h_2(f,g)$.
Making use of the formula~(\ref{ibp}), one gets
\begin{subequations}
\begin{align}
\widehat{h}_2(\xi) & = \int_2^t \!\!\int m^{1,0}_s e^{is\varphi_{++}(\xi,\eta)} \hat{f}(\xi - \eta) \hat{f}(\eta) d\eta ds 
= \int_2^t \!\!\int m^{1,0}_s \frac{1}{s} M^{-1,-1} \partial_\eta e^{is\varphi_{++}} \hat{f}(\xi - \eta) \hat{f}(\eta) d\eta ds
\nonumber \\
& \label{h2-1} = - \int_2^t \!\!\int \frac{1}{s} m^{-1,-2}_s e^{is\varphi_{++}} \hat{f}(\xi - \eta) \hat{f}(\eta) d\eta ds \\
& \label{h2-2} \;\;\;\;- \int_2^t \!\!\int \frac{1}{s} m^{0,-1}_s e^{is\varphi_{++}} \hat{f}(\xi - \eta) \partial_\eta \hat{f}(\eta) d\eta ds \;\;\;\;\mbox{ + \{ similar term \}}.
\end{align}
\end{subequations}
Applying Corollary~\ref{coro1} gives the estimates
\begin{equation*}
\begin{split} 
\|(\ref{h2-1})\|_2 & \lesssim \int_2^t \frac{1}{s} \left\| e^{is\Delta} T_{m^{-1,-2}_s} ( e^{  -is\Delta} f , e^{  -is\Delta} f ) \right\|_2 ds \\
& \lesssim \int_2^t \frac{1}{s} \sqrt{s} \| e^{  -is\Delta} f\|_{2} \|e^{  -is\Delta}f\|_{\infty} ds 
 \lesssim \int_2^t \eps^2   \frac{1}{s} \sqrt{s} \frac{1}{s} ds \lesssim    \eps^2 .\\[1em]
\|(\ref{h2-2})\|_2 & \lesssim \int_2^t \frac{1}{s} \left\| e^{is\Delta} T_{m^{0,-1}_s} ( e^{  -is\Delta} f , e^{  -is\Delta} x f ) \right\|_2 ds \\
& \lesssim \int_2^t \frac{1}{s} \|xf\|_2 \|e^{  -is\Delta} f\|_\infty ds 
 \lesssim \int_2^t  \eps^2  \frac{1}{s} \frac{1}{s} ds \lesssim  \eps^2  .
\end{split}
\end{equation*}

\subsection{Control of $x h_2(h,h)$ in $L^2$} 

\label{controlxh2}

Applying $\partial_\xi$ to $\widehat{h}_2(h,h)(\xi)$ yields
\begin{subequations}
\begin{align}
\partial_\xi \widehat{h}_2(\xi) = & \label{turtle1} \int_2^t \!\!\int m^{0,-1}_s e^{is\varphi_{++}(\xi,\eta)} \hat{h}(\xi - \eta) \hat{h}(\eta) d\eta ds \\
& \label{turtle2} + \int_2^t \!\!\int s m^{2,1}_s e^{is\varphi_{++}(\xi,\eta)} \hat{h}(\xi - \eta) \hat{h}(\eta) d\eta ds \\
& \label{turtle3} + \int_2^t \!\!\int m^{1,0}_s e^{is\varphi_{++}(\xi,\eta)} \hat{h}(\xi - \eta) \partial_\xi \hat{h}(\eta) d\eta ds .
\end{align}
\end{subequations}
Using~(\ref{ibp}) to integrate by parts, twice for~(\ref{turtle2}), and once for~(\ref{turtle1}) and~(\ref{turtle3}), we see that the above expressions are transformed into terms of the following types
\begin{subequations}
\begin{align}
& \label{xh2-1}\int_2^t \!\!\int \frac{1}{s} m^{-2,-3}_s e^{is\varphi_{++}(\xi,\eta)} \hat{h}(\xi - \eta) \hat{h}(\eta) d\eta \, ds \\
& \label{xh2-2} \int_2^t \!\!\int \frac{1}{s} m^{-1,-2}_s e^{is\varphi_{++}(\xi,\eta)} \hat{h}(\xi - \eta) \partial_\eta \hat{h}(\eta) d\eta \,ds \\
& \label{xh2-3} \int_2^t \!\!\int \frac{1}{s} m_s^{0,-1} e^{is\varphi_{++}(\xi,\eta)} \hat{h}(\xi - \eta) \partial_\eta^2 \hat{h}(\eta) d\eta \,ds \\
& \label{xh2-4} \int_2^t \!\!\int \frac{1}{s} m_s^{0,-1} e^{is\varphi_{++}(\xi,\eta)} \partial_\eta \hat{h}(\xi - \eta) \partial_\eta \hat{h}(\eta) d\eta \,ds 
\end{align}
\end{subequations}
Let us begin with~(\ref{xh2-1}). By Corollary~\ref{coro1},
\begin{equation*}
\begin{split} 
\|(\ref{xh2-1})\|_2 & \lesssim \int_2^t \frac{1}{s} s^{3/4} \|\Lambda_s^{-1/2} h\|_{2} \|e^{  -is\Delta} h\|_{\infty} ds \\
& \lesssim \int_2^t \frac{1}{s} s^{3/4} \|  h \|_{4/3} \|e^{-is\Delta} h\|_{\infty} ds 
 \lesssim \int_2^t s^{-1/4} \eps^2  \frac{1}{s} ds \lesssim \eps^2   .
\end{split}
\end{equation*}
For (\ref{xh2-2}), we have 
\begin{equation*}
\begin{split} 
\|(\ref{xh2-2})\|_2 & \lesssim \int_2^t \frac{1}{s} \left\| e^{is\Delta} T_{m^{-1,-2}_s} ( e^{  -is\Delta} h , e^{  -is\Delta} x h ) \right\|_2 ds \\
& \lesssim \int_2^t \frac{1}{s} s^{1/2} \|x h\|_{2}\|e^{  -is\Delta} h\|_{\infty}ds 
 \lesssim \int_2^t  \eps^2 \frac{1}{s} s^{1/2} \frac{1}{s} ds \lesssim \eps^2   .
\end{split}
\end{equation*}
Finally, estimating the term~(\ref{xh2-3}) only requires Corollary~\ref{coro1}: 
\begin{equation*}
\begin{split} 
\|(\ref{xh2-3})\|_2 & \lesssim \int_2^t \frac{1}{s} \left\| e^{is\Delta} T_{m^{0,-1}_s} ( e^{  -is\Delta} h , e^{  -is\Delta} x^2 h ) \right\|_2 ds \\
& \lesssim \int_2^t \frac{1}{s} \|x^2 h\|_2 \|e^{  -is\Delta}h\|_{\infty}ds 
 \lesssim \int_2^t  \eps^2 \frac{1}{s} s^{5/8} \frac{1}{s} ds \lesssim  \eps^2  .
\end{split}
\end{equation*}
The estimate of the term~(\ref{xh2-4}) reduces to the previous one with the help of Lemma~\ref{gagnir}.

\subsection{Control of $x^2 h_2(h,h)$ in $L^2$}

\label{blue}

Applying $\partial_\xi^2$ to $\widehat{h}_2(\xi)$
\begin{equation*}
\partial_\xi^2 \widehat{h}_2(\xi) = \partial_\xi^2 \int_2^t \!\!\int m^{1,0}_s e^{is\varphi_{++}(\xi,\eta)} \hat{h}(\xi - \eta) \hat{h}(\eta) d\eta ds
\end{equation*}
yields terms of the type
\begin{equation*}
\begin{split}
& \int_2^t \!\!\int \left( m^{-1,-2}_s +s m^{1,0}_s +s^2 m^{3,2}_s \right) e^{is\varphi_{++}(\xi,\eta)} \hat{h}(\xi - \eta) \hat{h}(\eta) d\eta ds \\
&  \int_2^t \!\!\int \left( m^{0,-1}_s +sm^{2,1}_s \right) e^{is\varphi_{++}(\xi,\eta)} \partial_\xi \hat{h}(\xi - \eta) \partial_\eta \hat{h}(\eta) d\eta ds \\
&  \int_2^t \!\!\int m^{1,0}_s e^{is\varphi_{++}(\xi,\eta)} \hat{h}(\xi - \eta) \partial_\eta^2 \hat{h}(\eta) d\eta ds .
\end{split}
\end{equation*}
Using~(\ref{ibp}), integrate by parts the above terms, twice if they contain a factor $s^2$, and once if they contain a factor $s$. Matters then reduce to estimating terms of the following types
\begin{subequations}
\begin{align}
& \label{xxh2-1} \int_2^t \!\!\int m^{-1,-2}_s e^{is\varphi_{++}(\xi,\eta)} \hat{h}(\xi - \eta) \hat{h}(\eta) d\eta ds \\
& \label{xxh2-2} \int_2^t \!\!\int m^{0,-1}_s e^{is\varphi_{++}(\xi,\eta)} \hat{h}(\xi - \eta) \partial_\eta \hat{h}(\eta) d\eta ds \\
& \label{xxh2-3} \int_2^t \!\!\int m^{1,0}_s e^{is\varphi_{++}(\xi,\eta)} \hat{h}(\xi - \eta) \partial_\eta^2 \hat{h}(\eta) d\eta ds \\
& \label{xxh2-4} \int_2^t \!\!\int m^{1,0}_s e^{is\varphi_{++}(\xi,\eta)} \partial_\eta \hat{h}(\xi - \eta) \partial_\eta \hat{h}(\eta) d\eta ds .
\end{align}
\end{subequations}
By Corollary~\ref{coro1},
\begin{equation*}
\begin{split} 
\|(\ref{xxh2-1})\|_2 & \lesssim \int_2^t \left\| e^{is\Delta} T_{m^{-1,-2}_s} ( e^{  -is\Delta} h , e^{  -is\Delta} h ) \right\|_2 ds \\
& \lesssim \int_2^t \sqrt{s} \|e^{-is\Delta} h \|_2 \|e^{  -is\Delta} h\|_\infty \,ds 
 \lesssim \int_2^t  \eps^2 s^{1/2} \frac{1}{s} ds \lesssim  \eps^2 \sqrt{t} .
\end{split}
\end{equation*}
The term~(\ref{xxh2-2}) is also easy to estimate, so we skip it, and consider next
\begin{equation*}
\begin{split}
\|(\ref{xxh2-3})\|_2 & \lesssim \int_2^t \left\| e^{is\Delta} T_{m^{1,0}_s } ( e^{  -is\Delta} x^2 h , e^{  -is\Delta} h ) \right\|_2 ds \\
& \lesssim \int_2^t \left\| x^2 h \right\|_2 \| e^{  -is\Delta} h \|_\infty ds 
 \lesssim \int_2^t  \eps^2 s^{5/8} \frac{1}{s} ds \lesssim  \eps^2 t^{5/8}.
\end{split}
\end{equation*}

Finally, the estimate of the term~(\ref{xxh2-4}) reduces to the previous one with the help of Lemma~\ref{gagnir}.

\subsection{Control of $ e^{it\Delta} h_2(h,h)$ in $L^\infty$}

\label{fashion}

The idea is to rewrite $h_2$ in the following fashion (we assume $t>3$, the adaptation to $2<t<3$ being obvious)
$$
h_2 = \int_2^{t-1} e^{is\Delta} H^1_2(s) \,ds + \int_{t-1}^t H^2_2(s) \,ds,
$$
where
$$
\left\{
\begin{array}{l} 
\widehat{H^1_2}(\xi,s) \overset{def}{=} \int m^{1,0}_s e^{is\widetilde{\varphi_{++}}(\xi,\eta)} \hat{h}(\xi - \eta) \hat{h}(\eta) \,d\eta \\
\widehat{H^2_2}(\xi,s) \overset{def}{=} \int m^{1,0}_s e^{is \varphi_{++} (\xi,\eta)} \hat{h}(\xi - \eta) \hat{h}(\eta) \,d\eta 
\end{array} 
\right.
$$
with
$$
\widetilde{\varphi_{++}}(\xi,\eta) = |\eta|^2 + |\xi-\eta|^2.
$$
Then 
\begin{subequations}
\begin{align}
\label{cat1} e^{-it\Delta} h_2(h,h) = & e^{-it\Delta} \int_2^{t-1} e^{is\Delta} H^1_2(s) \,ds \\
\label{cat2} & \;\;\;\;\;\;\;\;\;\; + e^{-it\Delta} \int_{t-1}^t H^2_2(s) \,ds.
\end{align}
\end{subequations}

\subsubsection{Estimate of~(\ref{cat1})}
The point is that $\partial_\eta \varphi_{++} = \partial_\eta \widetilde{\varphi_{++}}$, therefore the manipulations made in sections~\ref{h2hhl2} and~\ref{controlxh2} can also be performed on $H^1_2$.

On the one hand, it follows at once as in section~\ref{h2hhl2} that
\begin{equation}
\label{eagle}
\|H^1_2(s)\|_2 \lesssim \frac{\eps^2}{s\sqrt{s}} .
\end{equation}
On the other hand, just like we transformed $\partial_\xi \widehat{h}_2(h,h)$ into the terms~(\ref{xh2-1})-(\ref{xh2-4}), we can write
\begin{subequations}
\begin{align}
\label{springbok1} \partial_\xi \widehat{H^1_2}(\xi) = & \int \frac{1}{s} m^{-2,-3}_s e^{is\widetilde{\varphi_{++}}(\xi,\eta)} \hat{h}(\xi - \eta) \hat{h}(\eta) d\eta \\
\label{springbok2} & + \int \frac{1}{s} m^{-1,-2}_s e^{is\widetilde{\varphi_{++}}(\xi,\eta)} \hat{h}(\xi - \eta) \partial_\eta \hat{h}(\eta) d\eta  \\
\label{springbok3} & +  \int \frac{1}{s} m_s^{0,-1} e^{is\widetilde{\varphi_{++}}(\xi,\eta)} \hat{h}(\xi - \eta) \partial_\eta^2 \hat{h}(\eta) d\eta  \\
\label{springbok4} & + \int \frac{1}{s} m_s^{0,-1} e^{is\widetilde{\varphi_{++}}(\xi,\eta)} \partial_\eta \hat{h}(\xi - \eta) \partial_\eta \hat{h}(\eta) d\eta .
\end{align}
\end{subequations}
Let us focus for instance on~(\ref{springbok2}). Proceeding like in Section~\ref{h2hhl2} (but with these two important differences that we do not integrate in $s$ for the moment, and that the first $e^{is\Delta}$ does not appear any more, which allows estimates in Lebesgue spaces with indices lower than $2$), we get

\begin{equation}
\begin{split}
\|(\ref{springbok2})\|_{8/5} & \lesssim \frac{1}{s} \left\| T_{m^{-1,-2}_s} ( e^{-is\Delta} h , e^{  -is\Delta} x h  \right\|_{8/5} ds \\
& \lesssim \frac{1}{s} s^{1/2} \|x h\|_{2}\|e^{  -is\Delta} h\|_{8}ds 
 \lesssim \eps^2 \frac{1}{s} s^{1/2} \frac{1}{s^{3/4}} ds \lesssim \frac{\eps^2}{s^{5/4}} .
\end{split}
\end{equation}

Estimating similarly~(\ref{springbok1})~(\ref{springbok3})~(\ref{springbok4}), we get
\begin{equation}
\label{baldeagle}
\| x H^1_2 (s) \|_{8/5} \lesssim \frac{\eps^2}{s^{5/4}} .
\end{equation}
Putting together~(\ref{eagle}) and~(\ref{baldeagle}) gives
$$
\|H^1_2(s)\|_1 \lesssim \|H^1_2(s)\|_2 + \|x H^1_2 (s) \|_{8/5} \lesssim  \frac{\eps^2}{s^{5/4}}.
$$
Therefore 
\begin{equation*}
\begin{split}
\|(\ref{cat1})\|_\infty & = \left\| \int_2^{t-1} e^{i(s-t)\Delta} H^1_2(s)\,ds \right\|_\infty 
 \lesssim \int_2^{t-1} \frac{1}{t-s} \|H^1_2(s)\|_1 \,ds 
 \lesssim \eps^2 \int_2^{t-1} \frac{1}{t-s} \frac{1}{s^{5/4}} \,ds \lesssim \frac{\eps^2}{t}.
\end{split}
\end{equation*}

\subsubsection{Estimate of~(\ref{cat2})}

Proceeding as in Section~\ref{blue}, one gets
$$
\|\langle x \rangle^2 H^2_2(s) \|_2 \lesssim \eps^2 s^{-3/8} .
$$
Therefore,
\begin{equation*}
\begin{split}
\|(\ref{cat2})\|_\infty & = \left\| e^{-it\Delta} \int_{t-1}^t  H(s)\,ds \right\|_\infty 
 \lesssim \frac{1}{t} \int_{t-1}^t \|H(s)\|_1\,ds 
 \lesssim \eps^2 t^{-11/8}.
\end{split}
\end{equation*}

\section{Estimates on $h_2(g,f)$ and $h_2(f,g)$}

\label{h2fg}

As explained at the beginning of Section~\ref{h2hh}, we split the estimate of $h_2(f,f)$ into the estimate of $h_2(h,h)$, and the estimate of $h_2(f,g)$ and $h_2(g,f)$. The present section is dedicated to the latter kind of estimates.

The idea is that $g$ is a quadratic expression, essentially equal to a pseudo-product of $f$ with itself at time $t$. Therefore, as we will see shortly, $h_2(g,f)$ and $h_2(f,g)$ will be trilinear terms.

\subsection{Control of $h_2(g,f)$ and $h_2(f,g)$ in $L^2$}

This can be done as in Section~\ref{h2hhl2}.

\subsection{Decomposition in $(\xi,\eta,\sigma)$ space}

$h_2(f,g) + h_2(g,f)$ is given by
\begin{equation*}
\begin{split}
& h_2(f,g) + h_2(g,f) = \\
& \;\;\;\;\;\;\;\;\int_2^t \int \left[ \chi_s^{++S}(\xi,\eta) q(\xi,\eta) + \chi_s^{++S}(\xi,\xi-\eta) q(\xi,\xi-\eta) \right] \\
& \;\;\;\;\;\;\;\;\;\;\;\;\;\;\;\;\;\;\;\;\;\;\;\;\;\;\;\;\;\;\;\;\;\;\;\;\;\;\;\; \frac{q(\eta,\sigma) \chi_s^{++,T}(\eta,\sigma)}{\varphi_{++}(\eta,\sigma)} e^{is\varphi_{+++}(\xi,\eta,\sigma)} \widehat{f}(\sigma) \widehat{f}(\eta-\sigma) \widehat{f} (\xi-\eta) d\eta \, d\sigma \,ds \\
&\;\;\;\;\;\;\;\; + \int_2^t \int \left[ \chi_s^{++S}(\xi,\eta) q(\xi,\eta) + \chi_s^{++S}(\xi,\xi-\eta) q(\xi,\xi-\eta) \right] \\
& \;\;\;\;\;\;\;\;\;\;\;\;\;\;\;\;\;\;\;\;\;\;\;\;\;\;\;\;\;\;\;\;\;\;\;\;\;\;\;\;\frac{q(\eta,\sigma) \chi_s^{--,T}(\eta,\sigma)}{\varphi_{--}(\eta,\sigma)} e^{is\varphi_{+--}(\xi,\eta,\sigma)} \widehat{f}(\sigma) \widehat{f}(\eta-\sigma) \widehat{f} (\xi-\eta) d\eta \, d\sigma \,ds .\\
\end{split}
\end{equation*}
The phases $\varphi_{+++}$ and $\varphi_{+--}$ share (see Section~\ref{resonances}) the property that
$$
\mathscr{R} _{+++} = \mathscr{R} _{+--} = \{ \xi = \eta =\sigma = 0\}.
$$
This makes the two cases very similar; we will focus from now on on the $+++$ case,

\medskip

The symbols  which occur are  of the form $q(\xi,\eta) m^{0,0}(\xi,\eta) m^{-1,-2}(\eta,\sigma)$ or $q(\xi,\xi-\eta) m^{0,0}(\xi,\eta) m^{-1,-2}(\eta,\sigma)$.  
We now claim that it suffices to treat the case of a symbol of the form $m^{1,0}(\xi,\eta,\sigma) m^{0,0}(\xi,\eta) m^{-1,-2}(\eta,\sigma)$, in other words one can replace $q(\xi,\eta)$ (or $q(\xi,\xi-\eta)$) by $m^{1,0}(\xi,\eta,\sigma)$.  This is so since
\medskip

\noindent {\it a)} For $|(\xi,\eta)|\leq 1$, $q(\xi,\eta)$ satisfies homogeneous estimates in $(\xi,\eta,\sigma)$ of order 1: this follows from the linearity of $q$ in $\eta,\xi$.

\noindent {\it b)} For $|(\xi,\eta)|\geq 2$, $q(\xi,\eta)$ satisfies homogeneous estimates in $(\xi,\eta,\sigma)$ of order 0, simply because then $p$ is equal to one.

\noindent {\it c)}  We are left with $1 \leq |(\xi,\eta)|\leq 2$. Then $q$ satisfies homogeneous estimates in $(\xi,\eta,\sigma)$ except if $|\sigma| >> |(\xi,\eta)|$.

As a conclusion, $q(\xi,\eta)$ is of the form $m^{1,0}(\xi,\eta,\sigma)$ except if $1 \leq |(\xi,\eta)|\leq 2$ and $|\sigma| >> |(\xi,\eta)|$. This latter possibility is very simple to treat, since it is away from  the zero frequency, which is the main difficulty. We thus ignore it and consider in the following that


\begin{equation}
\label{modelfg}
\widehat{h}_2^{+++}(\xi) = 
\int_2^t \int m^{1,0}(\xi,\eta,\sigma) m^{0,0}(\xi,\eta) m^{-1,-2}(\eta,\sigma) e^{is\varphi_{+++}} \widehat{f}(\sigma) \widehat{f}(\eta-\sigma) \widehat{f} (\xi-\eta) d\eta \,d\sigma \,ds .
\end{equation}

Using cut-off functions $\chi^{+++,S}_s(\xi,\eta,\sigma)$ and $\chi^{+++,T}_s(\xi,\eta,\sigma)$ adapted to $\varphi_{+++}$, as explained in Section~\ref{partition}, one can decompose the integral defining $\widehat{h}_2^{+++}$. On the support of $\chi^T_s$, $\varphi_{+++}$ does not vanish, thus one can integrate by parts using 
$$
\frac{1}{i \varphi_{+++}} \partial_s e^{is\varphi_{+++}} = e^{is\varphi_{+++}}\;\;\;\mbox{written symbolically}\;\;\;M^{-2,-2} \partial_s e^{is\varphi_{+++}} = e^{is\varphi_{+++}}
$$
and obtain
\begin{equation*}
\begin{split}
& \widehat{h}_2^{+++}(\xi)  \\
& = \left. \int \chi^{+++,T}_s (\xi,\eta,\sigma)  m^{-1,-2}(\xi,\eta,\sigma) m^{0,0}(\xi,\eta) m^{-1,-2}(\eta,\sigma) e^{is\varphi_{+++}} \widehat{f}(\sigma) \widehat{f}(\eta-\sigma) \widehat{f} (\xi-\eta) d\eta\, d\sigma \,ds \right]_2^t\\
& \;\;\;\;\;\;- \int_2^t \int \chi^{+++,T}_s (\xi,\eta,\sigma)  m^{-1,-2}(\xi,\eta,\sigma) m^{0,0}(\xi,\eta) m^{-1,-2}(\eta,\sigma) \\
&\qquad\qquad \qquad  e^{is\varphi_{+++}}  \widehat{f}(\sigma) \widehat{f}(\eta-\sigma) \partial_s\widehat{f} (\xi-\eta) d\eta\, d\sigma \,ds+ \{\mbox{similar or easier  terms} \} \\
& \;\;\;\;\;\;+ \int_2^t \int \chi^{+++,S}_s(\xi,\eta,\sigma)  m^{1,0}(\xi,\eta,\sigma) m^{0,0}(\xi,\eta) m^{-1,-2}(\eta,\sigma)
 e^{is\varphi_{+++}} \widehat{f}(\sigma) \widehat{f}(\eta-\sigma) \widehat{f} (\xi-\eta) d\eta\, d\sigma \,ds \\
& \overset{def}{=} \widehat{h}_{2,1}^{+++} (\xi) + \widehat{h}_{2,2}^{+++} (\xi) + \widehat{h}_{2,3}^{+++} (\xi).
\end{split}
\end{equation*}
In particular  $\{\mbox{similar or easier  terms} \} $ include the case where the time derivative hits 
$\chi^{+++,T}_s (\xi,\eta,\sigma)  $ which gives a much simpler term and we will not 
detail it here. 
The terms~$h_{2,1}$ and $h_{2,2}$ will be estimated directly; as for~$h_{2,3}$, its decay is not strong enough to allow for direct estimates, and we will have to take advantage of the non-vanishing of $\partial_{\eta,\sigma} \varphi_{+++}$ on $\operatorname{Supp}\chi^{+++,S}_s(\xi,\eta,\sigma)$ and use the identity 
$$
\frac{1}{i s (\partial_{\eta,\sigma} \varphi_{+++})^2} \partial_{\eta,\sigma} \varphi_{+++} \cdot \partial_{\eta,\sigma} e^{is \varphi_{+++}} = e^{is \varphi_{+++}}
$$
that we write symbolically
\begin{equation}
\label{ibp2}
\frac{1}{s} M^{-1,-1} \partial_{\eta,\sigma} e^{is \varphi_{+++}} = e^{is \varphi_{+++}}.
\end{equation}

\subsection{Control of $x h_{2,1}^{+++}$ in $L^2$} 
First notice that $h_{2,1}^{+++}$ is the sum of one term evaluated at time $t$, and one term evaluated at time $2$ ; since the term evaluated at time $2$ is easy to estimate, we skip it, and focus in the following on the term corresponding to time $t$.

Applying $\partial_\xi$ to $h_{2,1}^{+++}$ gives terms of the type (the indices $j,k,l$ are always non-positive) 
\begin{subequations}
\begin{align}
\label{armadillo1}
& \int m^{j,j-1} m^{k,k} m^{l,l-1} e^{it\varphi_{+++}} \widehat{f}(\sigma) \widehat{f}(\eta-\sigma) \widehat{f}(\xi-\eta) d\eta \,d\sigma \;\;\;\;\mbox{with $j+k+l = -3$}\\
\label{armadillo2}
& \int t m^{0,-1} m^{0,0} m^{-1,-2} e^{it\varphi_{+++}} \widehat{f}(\sigma) \widehat{f}(\eta-\sigma) \widehat{f}(\xi-\eta) d\eta \,d\sigma\\
\label{armadillo3}
& \int m^{-1,-2} m^{0,0} m^{-1,-2} e^{it\varphi_{+++}} \widehat{f}(\sigma) \widehat{f}(\eta-\sigma) \partial_\xi \widehat{f}(\xi-\eta) d\eta \,d\sigma ;
\end{align}
\end{subequations}
all these terms can be estimated directly with Corollary~\ref{coro2}. For instance
\begin{equation*}
\begin{split}
\left\| (\ref{armadillo3}) \right\|_2 & = \left\| T_{m^{-1,-2}m^{0,0} m^{-1,-2}} (e^{-it\Delta} f , e^{-it\Delta} f , e^{-it\Delta} xf) \right\|_2 \\
& \lesssim t \left\| e^{-it\Delta} f \right\|_{16} \left\| e^{-it\Delta} f \right\|_{16} \left\| e^{-it\Delta} xf \right\|_{8/3} \\
& \lesssim  \eps^ t t^{-\frac{7}{8}} t^{-\frac{7}{8}} t^{-\frac{1}{4}} \left\| xf \right\|_{8/5} \lesssim \eps^2  t^{-1} \left\|\langle x \rangle^{\frac{3}{2}} f\right\|_{2} \lesssim \eps^3  t^{-\frac{1}{2}}.
\end{split}
\end{equation*}

\subsection{Control of $x^2 h_{2,1}^{+++}$ in $L^2$} 

Applying $\partial_\xi$ to $h_{2,1}^{+++}$ gives terms of the type (the indices $j$, $k$, $l$ are always non-positive)
\begin{subequations}
\begin{align}
& \int m^{j,j-1} m^{k,k} m^{l,l-1} e^{it\varphi_{+++}} \widehat{f}(\sigma) \widehat{f}(\eta-\sigma) \widehat{f}(\xi-\eta) d\eta \,d\sigma \;\;\;\;\mbox{with $j+k+l = -4$} \\
& \int t m^{j,j-1} m^{k,k} m^{l,l-1} e^{it\varphi_{+++}} \widehat{f}(\sigma) \widehat{f}(\eta-\sigma) \widehat{f}(\xi-\eta) d\eta \,d\sigma \;\;\;\;\mbox{with $j+k+l = -2$} \\
\label{bobcat3}
& \int  m^{j,j-1} m^{k,k} m^{l,l-1} e^{it\varphi_{+++}} \widehat{f}(\sigma) \widehat{f}(\eta-\sigma) \partial_\xi \widehat{f}(\xi-\eta) d\eta \,d\sigma \;\;\;\;\mbox{with $j+k+l = -3$} \\
& \int t^2 m^{1,0} m^{0,0} m^{-1,-2} e^{it\varphi_{+++}} \widehat{f}(\sigma) \widehat{f}(\eta-\sigma) \widehat{f}(\xi-\eta) d\eta \,d\sigma    \label{pb}\\
\label{bobcat5}
& \int t m^{0,-1} m^{0,0} m^{-1,-2} e^{it\varphi_{+++}} \widehat{f}(\sigma) \widehat{f}(\eta-\sigma) \partial_\xi \widehat{f}(\xi-\eta) d\eta \,d\sigma\\
\label{bobcat6}
& \int m^{-1,-2} m^{0,0} m^{-1,-2} e^{it\varphi_{+++}} \widehat{f}(\sigma) \widehat{f}(\eta-\sigma)  \partial_\xi^2 \widehat{f}(\xi-\eta) d\eta \,d\sigma ;
\end{align}
\end{subequations}
all these terms can be estimated directly with Corollary~\ref{coro2}, except for the last one, which requires a further manipulation; indeed, the Lebesgue exponents $\infty$, $\infty$, $2$ are not allowed by Corollary~\ref{coro2} for the arguments of the multilinear operator.

Thus one writes $ \partial_\xi^2 \widehat{f}(\xi-\eta) = - \partial_\eta \partial_\xi \widehat{f}(\xi-\eta)$, and integrates by parts in $\eta$. This yields terms of type~(\ref{bobcat3}) and~(\ref{bobcat5}), as well as
\begin{equation}
\label{meise}
\int m^{-1,-2} m^{0,0} m^{-1,-2} e^{it\varphi_{+++}} \widehat{f}(\sigma) \partial_\eta \widehat{f}(\eta-\sigma)  \partial_\xi \widehat{f}(\xi-\eta) d\eta \,d\sigma
\end{equation}
which can be estimated as follows
\begin{equation*}
\begin{split}
\left\| (\ref{meise}) \right\|_2 & = \left\| T_{m^{-1,-2}m^{0,0} m^{-1,-2}} (e^{-it\Delta} f , e^{-it\Delta} xf , e^{-it\Delta} xf) \right\|_2 \\
& \lesssim t \left\| e^{-it\Delta} f \right\|_{32} \left\| e^{-it\Delta} f \right\|_{32/14} \left\| e^{-it\Delta} xf \right\|_{32} \\
& \lesssim  t \left\| xf \right\|_{32/31} t^{-\frac{15}{16}} \left\| xf \right\|_{32/18} t^{-\frac{1}{8}} \epsilon t^{-\frac{15}{16}} \\
& \lesssim t \left\| \langle x^2 \rangle f \right\|_2 t^{-\frac{15}{16}} \left\| \langle |x|^{5/4} \rangle f \right\|_2 t^{-\frac{1}{8}} \epsilon t^{-\frac{15}{16}}\\
& \lesssim \epsilon^3 t \, t \, t^{-\frac{15}{16}} t^{1/4} t^{-\frac{1}{8}} t^{-\frac{15}{16}} \lesssim \epsilon^3 t^{1/8}.
\end{split}
\end{equation*}

\subsection{Control of $e^{-it\Delta} h_{2,1}^{+++}$ in $L^\infty$}

This control would be very easily obtained if pseudo-product operators were bounded with values in $L^\infty$. Since this is not the case, we use 
Sobolev inequality, namely 

\[
\begin{split}
|| e^{-it\Delta} h_{2,1}^{+++} ||_{L^\infty}  & \leq  || e^{-it\Delta} h_{2,1}^{+++} ||_{W^{1,8}}   \\
& \leq  || T_{m^{-1,-1}m^{0,0}m^{-1,-2}}( e^{-it\Delta} f ,  e^{-it\Delta} f ,  e^{-it\Delta} f)   ||_{L^8} \\
 & \leq   t || e^{-it\Delta} f ||_{L^{24}}^3  
 \leq  \eps^3 \frac1t  \frac1{t^{3/4}}
\end{split}
\]
where we have used that $ m^{-1,-2} + \xi m^{-1,-2}  = m^{-1,-1}   $. 

\subsection{Control of $xh_{2,2}^{+++}$ in $L^2$}

\label{xh2+++}

Applying $\partial_\xi$ to $\widehat{h}_{2,2}(\xi)$ gives terms of the type (the indices $j$, $k$, $l$ are always non-positive)
\begin{subequations}
\begin{align}
& \label{walrus1} \int_2^t \int m^{j,j-1} m^{k,k} m^{l,l-1} e^{is\varphi_{+++}}  \widehat{f}(\sigma) \widehat{f}(\eta-\sigma) \partial_s\widehat{f} (\xi-\eta) d\eta\, d\sigma \,ds \;\;\;\;\mbox{with $j+k+l = -3$} \\
& \label{walrus2} \int_2^t \int s m^{0,-1} m^{0,0} m^{-1,-2} e^{is\varphi_{+++}} \widehat{f}(\sigma) \widehat{f}(\eta-\sigma) \partial_s\widehat{f} (\xi-\eta) d\eta\, d\sigma \,ds \\
& \label{walrus3} \int_2^t \int m^{-1,-2} m^{0,0} m^{-1,-2} e^{is\varphi_{+++}} \widehat{f}(\sigma) \widehat{f}(\eta-\sigma) \partial_\xi \partial_s\widehat{f} (\xi-\eta) d\eta\, d\sigma \,ds.
\end{align}
\end{subequations}
We now transform~(\ref{walrus3}) by observing that $\partial_\xi \partial_s \widehat{f} (\xi-\eta) = - \partial_\eta \partial_s \widehat{f} (\xi-\eta)$ and integrating by parts in $\eta$. This gives terms of type~(\ref{walrus1})~(\ref{walrus2}) as well as
\begin{equation}
\label{walrus4}
\int_2^t \int m^{-1,-2} m^{0,0} m^{-1,-2} e^{is\varphi_{+++}} \widehat{f}(\sigma)  \partial_\eta\widehat{f}(\eta-\sigma)  \partial_s\widehat{f} (\xi-\eta) d\eta\, d\sigma \,ds.
\end{equation}
Actually, terms like \eqref{walrus4} can also come from the   $\{\mbox{similar or easier  terms} \}  $
in the definition of $h_{2,2}^{+++}$. 
Now the terms~(\ref{walrus1})~(\ref{walrus2})~(\ref{walrus4}) are easily estimated by using Corollary~\ref{coro2} and~(\ref{derf}). For instance
\begin{equation}
\begin{split}
\left\| (\ref{walrus1}) \right\|_2 & \leq \int_2^t \left\| e^{is\Delta} T_{m^{j,j-1} m^{k,k} m^{l,l-1}}(e^{-is\Delta}f,e^{-is\Delta}f,e^{-is\Delta}\partial_s f) \right\|_2 ds \\
& = \int_2^t \left\| T_{m^{j,j-1} m^{k,k} m^{l,l-1}}(e^{-is\Delta}f,e^{-is\Delta}f, Q(u,u) + Q(\bar u,\bar u) \right\|_2 ds \\
& \lesssim \int_2^t s^{3/2} \|e^{-is\Delta} f\|_8^4\,ds 
 \lesssim \int_2^t \eps^4 s^{3/2} (s^{-3/4})^4 \,ds \lesssim \eps^4.
\end{split}
\end{equation}

\subsection{Control of $x^2 h_{2,2}^{+++}$ in $L^2$}

\label{green}

We saw that $xh_{2,2}^{+++}$ can be reduced to terms of the type~(\ref{walrus1}) (\ref{walrus2}) (\ref{walrus4}). Now apply $\partial_\xi$ to (\ref{walrus1})~(\ref{walrus2})~(\ref{walrus4}), and, as in Section~\ref{xh2+++}, make sure by an integration by parts if necessary that an $s$ derivative and a $\xi$ derivative do not hit the same $f$. This gives terms of the types (the indices $j,k,l$ are always non-positive)
\begin{subequations}
\begin{align}
& \label{octopus1}\int_2^t \int m^{j,j-1} m^{k,k} m^{l,l-1} e^{is\varphi_{+++}}  \widehat{f}(\sigma) \widehat{f}(\eta-\sigma) \partial_s\widehat{f} (\xi-\eta) d\eta\, d\sigma \,ds \;\;\;\;\mbox{with $j+k+l = -4$} \\
& \label{octopus2}\int_2^t \int s^2 m^{1,0} m^{0,0} m^{-1,-2} e^{is\varphi_{+++}} \widehat{f}(\sigma) \widehat{f}(\eta-\sigma) \partial_s\widehat{f} (\xi-\eta) d\eta\, d\sigma \,ds \\
& \label{octopus3}\int_2^t \int m^{-1,-2} m^{0,0} m^{-1,-2} e^{is\varphi_{+++}}  \widehat{f}(\sigma)  \partial_\eta^2\widehat{f}(\eta-\sigma) \partial_s\widehat{f} (\xi-\eta) d\eta\, d\sigma \,ds \\
& \label{octopus5}\int_2^t \int s m^{j,j-1} m^{k,k} m^{l,l-1} e^{is\varphi_{+++}}  \widehat{f}(\sigma) \widehat{f}(\eta-\sigma) \partial_s\widehat{f} (\xi-\eta) d\eta\, d\sigma \,ds \;\;\;\;\mbox{with  $j+k+l = -2$} \\
& \label{octopus6}\int_2^t \int s m^{j,j-1} m^{k,k} m^{l,l-1} e^{is\varphi_{+++}}  \widehat{f}(\sigma)  \partial_\eta \widehat{f}(\eta-\sigma) \partial_s\widehat{f} (\xi-\eta) d\eta\, d\sigma \,ds \;\;\;\;\mbox{with  $j+k+l = -1$} \\
& \label{octopus7}\int_2^t \int m^{j,j-1} m^{k,k} m^{l,l-1} e^{is\varphi_{+++}} \widehat{f}(\sigma) \partial_\eta \widehat{f}(\eta-\sigma) \partial_s \widehat{f} (\xi-\eta) d\eta\, d\sigma \,ds \;\;\;\;\mbox{with $j+k+l = -3$} .
\end{align}
\end{subequations}
In order to bound the terms above, one should notice first that $e^{-it\Delta} \partial_s f= \alpha Q(u,u) + \beta Q(\bar u,\bar u)$, hence 
$$
\left\| e^{-it\Delta} \partial_s f \right\|_p \lesssim \eps^2 t^{\frac{2}{p}-2} \;\;\;\;\;\mbox{for $2\leq p<\infty$}.
$$
The estimates for~(\ref{octopus1})-(\ref{octopus7}) follow in a straightforward fashion using Corollary~\ref{coro2}, except for~(\ref{octopus3}). 

But for~(\ref{octopus3}), writing $\partial_\eta^2\widehat{f}(\eta-\sigma) = - \partial_\eta \partial_\sigma \widehat{f}(\eta-\sigma)$, and integrating by parts gives, in addition to already treated terms,
\begin{equation}
\label{loutre}
\int_2^t \int m^{-1,-2} m^{0,0} m^{-1,-2} e^{is\varphi_{+++}} \partial_\sigma \widehat{f}(\sigma)  \partial_\eta \widehat{f}(\eta-\sigma) \partial_s\widehat{f} (\xi-\eta) d\eta\, d\sigma \,ds,
\end{equation}
which can be treated directly:
\begin{equation}
\begin{split}
\left\| (\ref{loutre}) \right\|_2 & \lesssim \int_2^t \left\| T_{m^{-1,-2} m^{0,0} m^{-1,-2}} \left( e^{-it\Delta} xf, e^{-it\Delta} xf,e^{-it\Delta} \partial_s f\right) \right\|_2 \,ds \\
& \lesssim \int_2^t s  \left\|e^{-it\Delta} xf \right\|_{32/7}^2 \left\| e^{-it\Delta} \partial_s f \right\|_{16} \,ds 
 \lesssim \int_2^t s \left( s^{-9/16} \left\| xf \right\|_{32/25} \right)^2 \eps^2 s^{-15/8} \,ds \\
& \lesssim \int_2^t s \left( s^{-9/16} \left\| \langle x \rangle^{13/8} f \right\|_2 \right)^2 \eps^2 s^{-15/8} \, ds 
 \lesssim \eps^4 \int_2^t s \left( s^{-9/16} s^{5/8} \right)^2 s^{-15/8}\,ds 
  \lesssim \eps^4 t^{1/4}.
\end{split}
\end{equation}

\subsection{Control of $e^{-it\Delta} h_{2,2}^{+++}$ in $L^\infty$}

\label{sameidea}

Proceeding as in Section~\ref{fashion}, we rewrite
$$
h_{2,2}^{+++} = \int_2^{t-1} e^{is\Delta} H^1_{2,2}(s) \,ds + \int_{t-1}^t H^2_{2,2}(s) \,ds,
$$
where
$$
\left\{
\begin{array}{l} 
\widehat{H^1_{2,2}}(\xi,s) \overset{def}{=} \int m^{-1,-2}(\xi,\eta,\sigma) m^{0,0}(\xi,\eta) m^{-1,-2}(\eta,\sigma) e^{is\widetilde{\varphi}_{+++}}  \widehat{f}(\sigma) \widehat{f}(\eta-\sigma) \partial_s\widehat{f} (\xi-\eta) d\eta\, d\sigma  \\
\widehat{H^2_{2,2}}(\xi,s) \overset{def}{=} \int m^{-1,-2}(\xi,\eta,\sigma) m^{0,0}(\xi,\eta) m^{-1,-2}(\eta,\sigma) e^{is\varphi_{+++}}  \widehat{f}(\sigma) \widehat{f}(\eta-\sigma) \partial_s\widehat{f} (\xi-\eta) d\eta\, d\sigma  \\
\end{array} 
\right.
$$
with
$$
\widetilde{\varphi}_{+++}(\xi,\eta) =  |\xi-\eta|^2 + |\eta-\sigma|^2 + |\sigma|^2.
$$
Then 
\begin{subequations}
\begin{align}
\label{dog1} e^{-it\Delta} h_{2,2}^{+++} = & e^{-it\Delta} \int_2^{t-1} e^{is\Delta} H^1_{2,2}(s) \,ds \\
\label{dog2} & \;\;\;\;\;\;\;\;\;\; + e^{-it\Delta} \int_{t-1}^t H^2_{2,2}(s) \,ds.
\end{align}
\end{subequations}

\subsubsection{Estimate of~(\ref{dog1})}

On the one hand, one sees immediately that
\begin{equation}
\label{hawk}
\|H^1_{2,2}(s)\|_2 = \left\| T_{m^{-1,-2} m^{0,0} m^{-1,-2}} ( e^{-is\Delta} f , e^{-is\Delta} f, e^{-is\Delta}\partial_s f ) \right\|_2 \lesssim \frac{1}{s^2}
\end{equation}
On the other hand, proceeding as in Section~\ref{xh2+++}, we can write $\partial_\xi \widehat{H^1_{2,2}}$
as a sum of terms of the type (the indices $j,k,l$ are always non-positive)
\begin{subequations}
\begin{align}
& \label{seal1}  \int m^{j,j-1} m^{k,k} m^{l,l-1} e^{is\varphi_{+++}}  \widehat{f}(\sigma) \widehat{f}(\eta-\sigma) \partial_s\widehat{f} (\xi-\eta) d\eta\, d\sigma  \;\;\;\;\mbox{with $j+k+l = -3$} \\
& \label{seal2}  \int s m^{0,-1} m^{0,0} m^{-1,-2} e^{is\varphi_{+++}} \widehat{f}(\sigma) \widehat{f}(\eta-\sigma) \partial_s\widehat{f} (\xi-\eta) d\eta\, d\sigma  \\
& \label{seal3}  \int m^{-1,-2} m^{0,0} m^{-1,-2} e^{is\varphi_{+++}} \widehat{f}(\sigma)  \partial_\sigma\widehat{f}(\eta-\sigma)  \partial_s\widehat{f} (\xi-\eta) d\eta\, d\sigma .
\end{align}
\end{subequations}
Let us focus for instance on~(\ref{seal1}). It can be estimated in $L^{8/5}$ as follows:
\begin{equation}
\begin{split}
\|(\ref{seal1})\|_{8/5} & = \left\| T_{m^{j,j-1} m^{k,k} m^{l,l-1}} (e^{-is\Delta}f,e^{-is\Delta}f,e^{-is\Delta}\partial_s f) \right\|_{8/5}  \\
& \lesssim s^{3/2} \|e^{-is\Delta} f\|_{16} \|f\|_{16} \|\partial_s f\|_2  
 \lesssim \eps^4  s^{3/2} s^{-7/8} s^{-7/8} \frac{1}{s}  \lesssim \frac{\eps^4}{s^{5/4}}.
\end{split}
\end{equation}
Estimating similarly~(\ref{springbok2})~(\ref{springbok3}), we get
\begin{equation}
\label{redtailhawk}
\|x H^1_{2,2} (s) \|_{8/5} \lesssim \frac{\eps^4}{s^{5/4}} .
\end{equation}
Putting together~(\ref{hawk}) and~(\ref{redtailhawk}) gives
$$
\|H^1_{2,2}(s)\|_1 \lesssim \|H^1_{2,2}(s)\|_2 + \|x H^1_{2,2} (s) \|_{8/5} \lesssim  \frac{\eps^4}{s^{5/4}}.
$$
Therefore 
\begin{equation}
\begin{split}
\left\| (\ref{dog1}) \right\|_\infty & = \left\| \int_2^{t-1} e^{i(s-t)\Delta} H^1_{2,2}(s)\,ds \right\|_\infty 
 \lesssim \int_2^t \frac{1}{t-s} \|H_{2,2}(s)\|_1 \,ds \\
& \lesssim \eps^4\int_2^t \frac{1}{t-s} \frac{1}{s^{5/4}} \,ds 
 \lesssim \frac{\eps^4}{t}.
\end{split}
\end{equation}

\subsubsection{Estimate of~(\ref{dog2})}

Proceeding as in section~(\ref{green}), one sees that
$$
\| \langle x \rangle^2 H^2_{2,2} \|_2 \lesssim \frac{\eps^4}{\sqrt{t}} .
$$
Therefore,
\begin{equation}
\begin{split}
\left\| (\ref{dog2}) \right\|_\infty & = \left\| e^{-it\Delta} \int_{t-1}^t H_{2,2}^2(s) \,ds \right\|_\infty 
 \lesssim \frac{1}{t} \left\|\int_{t-1}^t H_{2,2}^2(s) \,ds \right\|_1 
 \lesssim \frac{\eps^4}{t\sqrt{t}} .
\end{split}
\end{equation}

\subsection{Control of $xh_{2,3}^{+++}$ in $L^2$}

\label{controlxh21}

Applying $\partial_\xi$ to $\widehat{h}_{2,3}^{+++}(\xi)$,
one gets terms of the types (the indices $j,k,l$ are always non-positive)
\begin{subequations}
\begin{align}
&\label{nightingale0} \int_2^t \int m^{1,0} m^{k,k} m^{l,l-1} e^{is \varphi_{+++}} \widehat{f}(\sigma) \widehat{f}(\eta-\sigma) \widehat{f} (\xi-\eta) d\eta\, d\sigma \,ds\;\;\;\;\mbox{with $k+l = -2$} \\
& \label{nightingale1} \int_2^t \int m^{j,j-1} m^{k,k} m^{l,l-1} e^{is \varphi_{+++}} \widehat{f}(\sigma) \widehat{f}(\eta-\sigma) \widehat{f} (\xi-\eta) d\eta\, d\sigma \,ds \;\;\;\;\mbox{with $j+k+l = -1$}\\
& \label{nightingale2} \int_2^t \int  m^{1,0} m^{0,0} m^{-1,-2}  e^{is \varphi_{+++}} \widehat{f}(\sigma) \widehat{f}(\eta-\sigma) \partial_\xi \widehat{f} (\xi-\eta) d\eta\, d\sigma \,ds \\
& \label{nightingale3} \int_2^t \int  s m^{2,1} m^{0,0} m^{-1,-2}  e^{is \varphi_{+++}} \widehat{f}(\sigma) \widehat{f}(\eta-\sigma) \widehat{f} (\xi-\eta) d\eta\, d\sigma \,ds.
\end{align}
\end{subequations}
The terms (\ref{nightingale1}) and (\ref{nightingale2}) can be estimated in a straightforward fashion using Corollary~\ref{coro2}. Using the identity~(\ref{ibp2}) to transform the last term above, we see that it can be reduced to terms of the type~(\ref{nightingale0}) (\ref{nightingale1}) (\ref{nightingale2}).  
Finally, using~(\ref{ibp2}) to transform~(\ref{nightingale0}) yields terms of type (the indices $j,k,l$ are always non-positive)
\begin{subequations}
\begin{align}
& \label{nightingale4} \int_2^t \int \frac{1}{s} m^{j,j-1} m^{k,k} m^{l,l-1} e^{is \varphi_{+++}} \partial_\sigma \widehat{f}(\sigma) \widehat{f}(\eta-\sigma) \widehat{f} (\xi-\eta) d\eta\, d\sigma \,ds\;\;\;\;\mbox{with $j+k+l = -2$} \\
& \label{nightingale5} \int_2^t \int \frac{1}{s} m^{j,j-1} m^{k,k} m^{l,l-1} e^{is \varphi_{+++}} \widehat{f}(\sigma) \widehat{f}(\eta-\sigma) \widehat{f} (\xi-\eta) d\eta\, d\sigma \,ds \;\;\;\;\mbox{with $j+k+l = -3$};
\end{align}
\end{subequations}
these terms can be estimated directly using Corollary~\ref{coro2}.

\subsection{Control of $e^{-it\Delta}h_{2,3}^{+++}$ in $L^\infty$} This can be done as in sections~\ref{fashion} and~\ref{sameidea}.

\subsection{Control of $x^2 h_{2,3}^{+++}$ in $L^2$} We saw in Section~\ref{controlxh21} that $xh_{2,3}^{+++}$ can be reduced to terms of the form (\ref{nightingale1}) (\ref{nightingale2}) (\ref{nightingale4}) (\ref{nightingale5}). Applying $\partial_\xi$ to (\ref{nightingale1}) (\ref{nightingale2}) (\ref{nightingale4}) (\ref{nightingale5}) gives
terms of the following types (the indices $j,k,l$ are always non-positive)
\begin{subequations}
\begin{align}
&\label{bull1} \int_2^t \int m^{j,j-1} m^{k,k} m^{l,l-1} e^{is \varphi_{+++}} \widehat{f}(\sigma) \widehat{f}(\eta-\sigma) \widehat{f} (\xi-\eta) d\eta\, d\sigma \,ds\;\;\;\;\mbox{with $j+k+l = -2$} \\
&\label{bull2} \int_2^t \int s m^{1,0} m^{k,k} m^{l,l-1} e^{is \varphi_{+++}} \widehat{f}(\sigma) \widehat{f}(\eta-\sigma) \widehat{f} (\xi-\eta) d\eta\, d\sigma \,ds\;\;\;\;\mbox{with $j+k+l = -1$} \\
&\label{bull3} \int_2^t \int m^{j,j-1} m^{k,k} m^{l,l-1} e^{is \varphi_{+++}} \widehat{f}(\sigma) \widehat{f}(\eta-\sigma) \partial_\xi \widehat{f} (\xi-\eta) d\eta\, d\sigma \,ds\;\;\;\;\mbox{with $j+k+l = -1$} \\
&\label{bull4} \int_2^t \int m^{1,0} m^{k,k} m^{l,l-1} e^{is \varphi_{+++}} \widehat{f}(\sigma) \widehat{f}(\eta-\sigma) \partial_\xi  \widehat{f} (\xi-\eta) d\eta\, d\sigma \,ds\;\;\;\;\mbox{with $j+k+l = -2$} \\
&\label{bull5} \int_2^t \int s m^{2,1} m^{0,0} m^{-1,-2} e^{is \varphi_{+++}} \widehat{f}(\sigma) \widehat{f}(\eta-\sigma)  \partial_\xi \widehat{f} (\xi-\eta) d\eta\, d\sigma \,ds \\
&\label{bull6} \int_2^t \int  m^{1,0} m^{0,0} m^{-1,-2} e^{is \varphi_{+++}} \widehat{f}(\sigma) \widehat{f}(\eta-\sigma) \partial_\xi^2 \widehat{f} (\xi-\eta) d\eta\, d\sigma \,ds \\
&\label{bull7} \int_2^t \int \frac{1}{s} m^{j,j-1} m^{k,k} m^{l,l-1} e^{is \varphi_{+++}} \partial_\sigma \widehat{f}(\sigma) \widehat{f}(\eta-\sigma) \widehat{f} (\xi-\eta) d\eta\, d\sigma \,ds\;\;\;\;\mbox{with $j+k+l = -3$} \\
&\label{bull8} \int_2^t \int  m^{j,j-1} m^{k,k} m^{l,l-1} e^{is \varphi_{+++}} \partial_\sigma \widehat{f}(\sigma) \widehat{f}(\eta-\sigma) \widehat{f} (\xi-\eta) d\eta\, d\sigma \,ds\;\;\;\;\mbox{with $j+k+l = -1$} \\
&\label{bull9} \int_2^t \int \frac{1}{s} m^{j,j-1} m^{k,k} m^{l,l-1} e^{is \varphi_{+++}} \partial_\sigma \widehat{f}(\sigma) \widehat{f}(\eta-\sigma) \partial_\xi \widehat{f} (\xi-\eta) d\eta\, d\sigma \,ds\;\;\;\;\mbox{with $j+k+l = -2$} \\
&\label{bull10} \int_2^t \int \frac{1}{s} m^{j,j-1} m^{k,k} m^{l,l-1} e^{is \varphi_{+++}}  \widehat{f}(\sigma) \widehat{f}(\eta-\sigma) \widehat{f} (\xi-\eta) d\eta\, d\sigma \,ds\;\;\;\;\mbox{with $j+k+l = -4$} \\
&\label{bull11} \int_2^t \int m^{j,j-1} m^{k,k} m^{l,l-1} e^{is \varphi_{+++}}  \widehat{f}(\sigma) \widehat{f}(\eta-\sigma) \widehat{f} (\xi-\eta) d\eta\, d\sigma \,ds\;\;\;\;\mbox{with $j+k+l = -2$} .
\end{align}
\end{subequations}
These terms can be bounded in a similar way to all the estimates already performed, except for two of them:~(\ref{bull5}) and~(\ref{bull6}). Since the former can be reduced to the latter by integration by parts using~(\ref{ibp2}), we shall focus on~(\ref{bull6}), the difficulty being that the $L^\infty \times L^\infty \times L^2 \rightarrow L^2$ estimate does not hold in general for flag singularity paraproducts; to go around it, we shall use $(ii)$ in Theorem~\ref{FS}.

First observe that the case where $|\xi|\lesssim |\eta,\sigma|$ can be easily dealt with, for then the symbol $m^{1,0}(\xi,\eta,\sigma) m^{0,0}(\xi,\eta) m^{-1,-2}(\eta,\sigma)$ becomes $m^{0,-2}(\xi,\eta,\sigma) m^{0,0}(\xi,\eta)$. Thus we shall assume that $|\xi|>>|\eta,\sigma|$. Next, by  symmetry, it is possible to assume that in the integral defining~(\ref{bull6}), $|\sigma| \gtrsim |\eta-\sigma|$. Thus it suffices to consider the case $|\xi|>>|\sigma| \gtrsim |\eta-\sigma|$. As usual, this can be ensured by adding a cut-off function, which we denote $\chi(\xi,\eta,\sigma)$. Finally notice that the condition $|\sigma| \gtrsim |\eta-\sigma|$ imposes $|\sigma| \gtrsim \frac{1}{\sqrt{s}}$ on the support of $m^{-1,-2} (\sigma, \eta-\sigma)$. We now decompose
\begin{equation}
\label{coucou}
\begin{split}
& \int_2^t \int  m^{1,0} m^{0,0} m^{-1,-2} \chi(\eta,\sigma) e^{is \varphi_{+++}} \widehat{f}(\sigma) \widehat{f}(\eta-\sigma) \partial_\xi^2 \widehat{f} (\xi-\eta) d\eta\, d\sigma \,ds \\
& \qquad \qquad  = \int_2^t \int  \sum_{2^j \gtrsim \frac{1}{\sqrt{s}}} 2^{-j} m^{1,0} m^{0,0} m^{0,-1} \chi(\eta,\sigma) e^{is \varphi_{+++}}\\
&\hskip 40mm \times  \theta \left( \frac{\sigma}{2^j} \right) \widehat{f}(\sigma) \Theta \left( \frac{\eta-\sigma} {4\cdot 2^j} \right) \widehat{f}(\eta-\sigma) \partial_\xi^2 \widehat{f} (\xi-\eta) d\eta\, d\sigma \,ds.
\end{split}
\end{equation}
This can be estimated by $(ii)$ of Theorem~\ref{FS}:
\begin{equation*}
\begin{split}
\left\|(\ref{coucou}) \right\|_2 & \lesssim \int_2^t \sum_{2^j \gtrsim \frac{1}{\sqrt{s}}} 2^{-j} \left\| T_{m^{1,0} m^{0,0} m^{0,-1}} \left( P_j e^{is\Delta} f , P_{<j+2} e^{is\Delta}f, e^{is\Delta}x^2 f \right) \right\|_2 \,ds \\
& \lesssim \int_2^t \sum_{2^j \gtrsim \frac{1}{\sqrt{s}}} 2^{-j} \left\| e^{is\Delta} f \right\|_\infty \left\| e^{is\Delta} f \right\|_\infty \left\| x^2 f \right\|_2 \,ds 
 \lesssim \int_2^t  \eps^3 \sqrt{s} \frac{1}{s} \frac{1}{s} s \,ds 
 \lesssim  \eps^3 \sqrt{t}.
\end{split}
\end{equation*}

\section{Estimates on $h_3$}

\label{sectionh3}

From its definition~ in section~\ref{dolphin}, we see that $\widehat{h}_3$ can be written as
$$
\widehat{h}_3(\xi) = \widehat{h}_3^{+++}(\xi) + \widehat{h}_3^{+--}(\xi) +\widehat{h}_3^{---}(\xi) +\widehat{h}_3^{-++}(\xi) 
$$
with
\begin{equation}
\label{fox}
\widehat{h}_3^{\pm \pm \pm}(\xi) = \int_2^t  \int q(\eta,\sigma) m^{-1,-2}(\xi,\eta)  e^{is\varphi_{\pm\pm\pm}} \widehat{f}(\sigma) \widehat{f}(\eta-\sigma) \widehat{f}(\xi-\eta) \,d\eta\,d\sigma\,ds.
\end{equation}
Observe (as in Section~\ref{h2fg}) that the symbol $q(\eta,\sigma)  m^{-1,-2}(\xi,\eta)$ can be written $m^{1,0} (\xi,\eta,\sigma) m^{-1,-2}(\xi,\eta)$.

\subsection{The cases $+++$, $+--$ and $---$}

Notice (see Section~\ref{cubic}) that the three phases $\varphi = \varphi_{+++},\varphi_{+--},\varphi_{---}$ correspond to a space-time resonant set $\mathscr{R}  = \{ \varphi= 0\} \cup \{\partial_{\eta,\sigma} \varphi = 0  \} = \{ \xi =\eta =0\}$. Thus one can proceed as in Section~\ref{h2fg} to derive the desired estimates in most cases. Only one term has to be treated in a different way. It occurs when estimating $x^2 h_2^{+++}$ (we focus from now on on the $+++$ case), corresponds to~(\ref{bull6}), and reads
\begin{equation*}
\int_2^t \int  m^{1,0}(\xi,\eta,\sigma) m^{-1,-2}(\xi,\eta) e^{is \varphi_{+++}} \widehat{f}(\sigma) \widehat{f}(\eta-\sigma) \partial_\xi^2 \widehat{f} (\xi-\eta) d\eta\, d\sigma \,ds.
\end{equation*}
In order to estimate it, one has to distinguish between the cases $|\eta| >> |\xi-\eta|$ and $|\eta| \lesssim |\xi-\eta|$. Since they are fairly similar, we shall focus on the former; as usual this is ensured by adding a cut-off function $\chi$ which localizes frequencies to this set, thus we now consider
\begin{equation}
\label{mesange}
\int_2^t \int  m^{1,0}(\xi,\eta,\sigma) m^{-1,-2}(\xi,\eta) \chi(\xi,\eta) e^{is \varphi_{+++}} \widehat{f}(\sigma) \widehat{f}(\eta-\sigma) \partial_\xi^2 \widehat{f} (\xi-\eta) d\eta\, d\sigma \,ds.
\end{equation}
Observe that $|\eta| >> |\xi-\eta|$ implies $|\xi-\eta| << |\xi|$; and also that the support condition on $m^{-1,-2}_s(\xi,\eta)$ implies $|\xi| \gtrsim \frac{1}{\sqrt{s}}$. Therefore, we can estimate, with the help of Bernstein's lemma
\begin{equation*}
\begin{split}
\left\| (\ref{mesange}) \right\|_2 & \lesssim \sum_j \left\| P_j (\ref{mesange}) \right\|_2 \\
& = \left\| \sum_j P_j \int_2^t \int_2^t T_{m^{1,0} m^{-1,-2} \chi(\xi,\eta)} \left(e^{is\Delta} f\,,\,e^{is\Delta} f\,,\,P_{<j+1} e^{is\Delta} x^2 f \right) \,ds \right\|_2 \\
& \lesssim \int_2^t \sum_{2^j \gtrsim \frac{1}{\sqrt{s}}} 2^{-j} \left\|  T_{m^{1,0} m^{0,-1}} \left(e^{is\Delta} f\,,\,e^{is\Delta} f\,,\,P_{<j+1} e^{is\Delta} x^2 f \right) \right\|_2 \,ds \\
& \lesssim \int_2^t \sum_{2^j \gtrsim \frac{1}{\sqrt{s}}} 2^{-j} \left\| e^{is\Delta} f \right\|_{16} \left\| e^{is\Delta} f \right\|_{16} \left\| e^{is\Delta} P_{<j+1} x^2 f \right\|_{8/3} \,ds \\
& \lesssim  \eps^2 \int_2^t \sum_{2^j \gtrsim \frac{1}{\sqrt{s}}} 2^{-j} s^{-7/8} s^{-7/8} 2^{j/4} \|x^2 f\|_2 \,ds \\
& \lesssim  \eps^3 \int_2^t s^{3/8} s^{-7/8} s^{-7/8} s\,ds 
 \lesssim  \eps^3 t^{5/8}.
\end{split}
\end{equation*}
 
\subsection{The case $-++$}

In the case of $\varphi_{-++}$, the space-time resonant set $\mathscr{R} ^{-++} = \{ \xi = \sigma = \frac{1}{2} \eta \} $ (see~(\ref{cubic}) is not reduced to the origin. Using an appropriate (smooth, homogeneous of degree 0) cut-off function $\chi_{-++}$, one can restrict the problem to a neighbourhood of $\mathscr{R} ^{-++}$, the rest being treated as in Section~\ref{h2fg}. Furthermore, this neigbourhood is chosen such that $\xi$, $\eta$, $\sigma$ essentially have the same size.
This has the advantage of canceling the flag singularity, in other words the above symbol $\chi_{-++}(\xi,\eta,\sigma) m^{1,0}(\xi,\eta,\sigma) m^{-1,-2}(\xi,\eta)$ can be replaced by $\chi_{-++}(\xi,\eta,\sigma) m^{0,-2}(\xi,\eta,\sigma)$. 

In the following of this section, we will thus consider the term obtained after restricting $\xi$, $\eta$, $\sigma$ to a neighbourhood of $\mathscr{R} ^{-++}$:
\begin{equation}
\label{tiger}
\widehat{\widetilde{h}}_3^{-++}(\xi) = 
\int_2^t \int \chi_{-++}(\xi,\eta,\sigma) m^{0,-2}(\xi,\eta,\sigma) e^{is\varphi_{-++}} \widehat{f}(\sigma) \widehat{f}(\eta-\sigma) \widehat{f}(\xi-\eta) \,d\eta\,d\sigma\,ds.
\end{equation}

\subsection{Control of $\widetilde{h}_3^{-++}$ in $L^2$}

Immediate.

\subsection{Control of $x \widetilde{h}_3^{-++}$ in $L^2$}

\label{xh3l2}

Applying $\partial_\xi$ to $\widehat{\widetilde{h}}_3^{-++}$ gives
\begin{subequations}
\begin{align}
\label{sparrow2}
&  \int_2^t \int \chi_{-++}(\xi,\eta,\sigma) m^{-1,-3} (\xi,\eta,\sigma) e^{is\varphi_{-++}} \widehat{f}(\sigma) \widehat{f}(\eta-\sigma) \widehat{f}(\xi-\eta) \,d\eta\,d\sigma\,ds \\
\label{sparrow1} & +\int_2^t \int \chi_{-++}(\xi,\eta,\sigma) m^{0,-2}(\xi,\eta,\sigma) e^{is\varphi_{-++}} \partial_\xi \widehat{f}(\sigma) \widehat{f}(\eta-\sigma) \widehat{f}(\xi-\eta) \,d\eta\,d\sigma\,ds \\
\label{sparrow3} & + \int_2^t \int \chi_{-++}(\xi,\eta,\sigma) m^{0,-2}(\xi,\eta,\sigma) s \partial_\xi \varphi_{-++} e^{is\varphi_{-++}}  \widehat{f}(\sigma) \widehat{f}(\eta-\sigma) \widehat{f}(\xi-\eta) \,d\eta\,d\sigma\,ds .
\end{align}
\end{subequations}
The term~(\ref{sparrow2}) and (\ref{sparrow1}) can be estimated without any difficulty by Corollary~\ref{coro2}. 
For~(\ref{sparrow3}), we use the following relation:
$$
\partial_\xi \varphi_{-++} = -2 \partial_\eta \varphi_{-++} - \partial_\sigma \varphi_{-++}.
$$
Substituting the above right-hand side for $\partial_\xi \varphi_{-++}$ in (\ref{sparrow3}), we can integrate this term by parts using the relations
$$
s \partial_\eta \varphi_{-++} e^{is\varphi_{-++}} = \frac{1}{i} \partial_\eta e^{is\varphi_{-++}} \;\;\;\;\;\;\;\mbox{and}\;\;\;\;\;\;\; s \partial_\sigma \varphi_{-++} e^{is\varphi_{-++}} = \frac{1}{i} \partial_\sigma e^{is\varphi_{-++}},
$$
and the result is terms of the form~(\ref{sparrow2}) and~(\ref{sparrow1})

\subsection{Control of $e^{-it\Delta} \widetilde{h}_3^{-++}$ in $L^\infty$}

Our strategy will be the following: by the standard dispersive estimate, the decay of $e^{-it\Delta} \widetilde{h}_3^{-++}$ in $L^\infty$ follows from a bound on $\left\| \widetilde{h}_3^{-++} \right\|_1$. The quantity whose control was obtained in the previous paragraph, namely $\left\| x \widetilde{h}_3^{-++} \right\|_2$ barely fails to control $\left\| \widetilde{h}_3^{-++} \right\|_1$, but it will suffice to obtain a control of this weighted norm with a Lebesgue index slightly smaller than $2$.

To this we now turn: we will prove that $\left\| x \widetilde{h}_3^{-++} \right\|_{8/5}$ remains bounded, and, as explained above, this will give us the desired result since
$$
\left\| e^{-it\Delta} \widetilde{h}_3^{-++} \right\|_\infty \lesssim \frac{1}{t} \left\| \widetilde{h}_3^{-++} \right\|_1 \lesssim \frac{1}{t} \left( \left\| \widetilde{h}_3^{-++} \right\|_2 + \left\| x  \widetilde{h}_3^{-++} \right\|_{8/5} \right) .
$$

We saw in the previous paragraph that $x \widetilde{h}_3^{-++}$ can be written as a sum of the terms of the type~(\ref{sparrow2}) or~(\ref{sparrow1}). We will show how to obtain a bound for terms of the type~(\ref{sparrow1}), the case of~(\ref{sparrow2}) can be treated in an identical fashion.
 
Next observe that one can write
$$
B_{\chi_{-++}(\xi,\eta,\sigma) m^{0,-2}(\xi,\eta,\sigma) } = P_{<-\frac{1}{2}\log s} B_{m^{0,0}(\xi,\eta,\sigma)} + \sum_{-\frac{1}{2} \log s < j } \inf(1,2^{-2j}) P_j B_{m^{0.0}(\xi,\eta,\sigma)} .
$$
Therefore, using in addition to the traditional arguments the point $(iv)$ of Lemma~\ref{boundlin} gives
\begin{equation*}
\begin{split}
\left\| x  \widetilde{h}_3^{-++} \right\|_{8/5} & = \left\| \int_0^t  e^{is\Delta} B_{\chi_{-++}(\xi,\eta,\sigma) m^{0,-2}(\xi,\eta,\sigma) } \left( e^{-is\Delta} f , e^{-is\Delta} f , e^{-is\Delta} x f \right) \right\|_{8/5} \\
& \lesssim \int_0^t  \Big[ 
\left\| P_{<-\frac{1}{2}\log s} e^{is\Delta} B_{\chi_{-++}(\xi,\eta,\sigma) m^{0,0}(\xi,\eta,\sigma) } \left( e^{-is\Delta} f , e^{-is\Delta} f , e^{-is\Delta} x f \right) \right\|_{8/5} \\
& \;\;\;\;\;\;\;\;  + \sum_{j>-\frac{1}{2} \log s} \inf(1,2^{-2j}) \left\| P_j B_{m^{0.0}(\xi,\eta,\sigma)}\left( e^{-is\Delta} f , e^{-is\Delta} f , e^{-is\Delta} x f \right) \right\|_{8/5} \Big]
\, ds \\
& \lesssim \int_0^t \Big[ \left\| u(s) \right\|_{16} \left\| u(s) \right\|_{16} \left\| xf \right\|_2  +\!\!\!\!\sum_{j>-\frac{1}{2} \log s} \!\!\!\! \inf(1,2^{-2j}) 2^{j/4} t^{1/8} \left\| u(s) \right\|_{16} \left\| u(s) \right\|_{16} \left\| xf \right\|_2  \Big] ds \\
& \lesssim \eps^3 \int_0^t \Big[  s^{-7/8} s^{-7/8} + \sum_{j>-\frac{1}{2} \log s} \inf(1,2^{-2j}) 2^{j/4} s^{1/8} s^{-7/8} s^{-7/8} \Big]  ds  \lesssim  \eps^3  .
\end{split}
\end{equation*} 

\subsection{Control of $x^2 \widetilde{h}_3^{-++}$ in $L^2$}

We saw in Section~\ref{xh3l2} that $x   \widetilde{h}_3^{-++}  $ could be reduced to the terms~(\ref{sparrow2}) and (\ref{sparrow1}). Applying $\partial_\xi$ to these two terms gives expressions of the following types
\begin{subequations}
\begin{align}
 \label{seagull1} &  \int_2^t \int m^{-2,-4}(\xi,\eta,\sigma) e^{is\varphi_{-++}} \widehat{f}(\sigma) \widehat{f}(\eta-\sigma) \widehat{f}(\xi-\eta) \,d\eta\,d\sigma\,ds \\
\label{seagull2} & \int_2^t \int m^{-1,-3}(\xi,\eta,\sigma) e^{is\varphi_{-++}} \widehat{f}(\sigma) \widehat{f}(\eta-\sigma)  \partial_\xi \widehat{f}(\xi-\eta)\,d\eta\,d\sigma\,ds \\
\label{seagull3} &  \int_2^t s \int m^{0,-2}(\xi,\eta,\sigma) e^{is\varphi_{-++}} \widehat{f}(\sigma)  \widehat{f}(\eta-\sigma) \widehat{f}(\xi-\eta) \,d\eta\,d\sigma\,ds \\
\label{seagull4} &  \int_2^t s \int m^{1,-1}(\xi,\eta,\sigma) e^{is\varphi_{-++}} \widehat{f}(\sigma)  \widehat{f}(\eta-\sigma) \partial_\xi \widehat{f}(\xi-\eta) \,d\eta\,d\sigma\,ds \\
\label{seagull5} &  \int_2^t \int m^{0,-2}(\xi,\eta,\sigma) e^{is\varphi_{-++}} \widehat{f}(\sigma)  \widehat{f}(\eta-\sigma) \partial_\xi^2 \widehat{f}(\xi-\eta) \,d\eta\,d\sigma\,ds .
\end{align}
\end{subequations}
All these expressions can be estimated directly with the help of Corollary~\ref{coro2}.

\appendix
\renewcommand{\theequation}{\Alph{section}.\arabic{equation}}
\setcounter {equation} {0}
\section{Proof of Theorem~\ref{FS}}

\label{appendix}

We shall only prove $(i)$ in Theorem~\ref{FS}: if $m$ is of flag singularity type with degree 0, then the operator
$$
T_m : L^p \times L^q \times L^r \rightarrow  L^s
$$
is bounded for 
$$
\frac{1}{s} = \frac{1}{p} + \frac{1}{q} + \frac{1}{r}  \quad \mbox{if $1<p,q,r,s < \infty$}
$$
with a bound  less than a multiple of $\|m\|_{FS}$, and this result remains true if $s=2$ and only 
 one of $p,q,r$ is taken equal to $\infty$.

The proof of $(ii)$ follows the steps of $(i)$, but is much simpler, thus we shall skip it.

\begin{remark}
It would be of particular interest for the PDE problem which is the heart of the present paper to obtain estimates of the type $L^\infty \times L^\infty \times L^2 \rightarrow L^2$. This set of Lebesgue indices is not covered by Theorem~\ref{FS}, unless a projection $P_0$ on a band of frequencies is added. To see that boundedness for this choice of spaces does not hold in general, take $B$ a bilinear Coifman-Meyer operator, and form the flag singularity pseudo-product operator
$$
T \;\overset{def}{:} \; (f_1,f_2,f_3) \;\mapsto \; B(f_1,f_2)f_3.
$$
The operator $T$ is bounded from $L^\infty \times L^\infty \times L^2$ to $L^2$ if and only if $B$ is bounded from $L^\infty \times L^\infty$ to $L^\infty$; but this last property is not true for general Coifman-Meyer operators.
\end{remark}

We now start with the proof of Theorem~\ref{FS}:

\noindent 
\underline{Step 1: partition of the $(\xi,\eta,\sigma)$ plane}
By definition of a flag singularity with degree 0, the symbol $m$ can be written
$$
m^{III}(\xi,\eta,\sigma) m^{II}_1(\eta,\xi) m^{II}_2(\eta,\sigma) ,
$$
with $m^{III}$, $m^{II}_1$ and $m^{II}_2$ of Coifman-Meyer type. First observe that there are certain regions of the $(\xi,\eta,\sigma)$ plane where $m$ satisfies the Coifman-Meyer estimates~(\ref{butterfly}); hence the Coifman-Meyer theorem applies, and the desired estimate is proved. Thus, using a (homogeneous of degree 0, smooth away from 0) cut-off function, we can reduce the problem to the regions where the Coifman-Meyer estimate~(\ref{butterfly}) does not hold for $m$, namely 
$$
A_1 \cup A_2 \overset{def}{=} \{ |\xi| + |\eta| \leq \epsilon |\sigma| \} \cup \{ |\eta| + |\sigma| \leq \epsilon |\xi| \},
$$
where $\epsilon$ is a small constant. We further observe that on $A_1$ (respectively: on $A_2$), $m^{II}_2$ (respectively $m^{II}_1$) satisfies the Coifman-Meyer estimate in $(\xi,\eta,\sigma)$. Now choose cut-off functions $\chi_{A_1}(\xi,\eta,\sigma)$ and $\chi_{A_2}(\xi,\eta,\sigma)$ which have homogeneous bounds of degree 0, and localize respectively near $A_1$ and $A_2$. More precisely, we choose $\chi_{A_1}$ such that
\begin{equation*}
T_{\chi_{A_1}(\xi,\eta,\sigma)} (f_1,f_2,f_3) = \sum_{k} P_{<k-100} \left( P_k f_1 P_k f_2 \right) P_{<k-100} f_3,
\end{equation*}
and similarly for $\chi_{A_2}$. With the help of these cutoff functions, we can reduce matters to symbols of the two following types:
\begin{equation}
\begin{split}
\label{lion}
& \chi_{A_1}(\xi,\eta,\sigma) m^{III}(\xi,\eta,\sigma) m^{II}_1(\eta,\xi) \\
& \chi_{A_2}(\xi,\eta,\sigma) m^{III}(\xi,\eta,\sigma) m^{II}_2(\eta,\sigma) \\
\end{split}
\end{equation}
Observe that, $\langle\cdot , \cdot \rangle$ denoting the standard scalar product,
$$
\langle T_{\mu(\xi,\eta,\sigma) \nu(\eta,\sigma)} (f_1,f_2,f_3) \,,\,f_4 \rangle = \langle T_{\mu(\sigma,\eta,\xi) \nu(\eta,\xi)} (f_4,\bar{f_3}(-\cdot),\bar{f_2}(-\cdot))\,,\,f_1 \rangle .
$$
thus estimates for one of the above symbols can be deduced from the other by duality if all the Lebesgue indices are finite.

We focus from now on symbols of the first type in~(\ref{lion}). 

\bigskip

\noindent
\underline{Step 2: series expansion of $m^{III}$} 
Let us expand in series the symbol $m^{III}(\xi,\eta,\sigma)$ around $(\eta,\xi)=0$. One gets
$$
m^{III}(\xi,\eta,\sigma) = \sum_{|\alpha|=0}^{M-1} \Phi_\alpha(\sigma) (\eta,\xi)^\alpha + R(\xi,\eta,\sigma),
$$ 
where $\Phi_\alpha(\sigma) = \frac{\partial^\alpha}{\partial(\xi,\eta)^\alpha} m^{III}(0,0,\sigma)$ and the remainder $R$ satisfies 
\begin{equation}
\label{boundR}
\left| \partial_{\eta,\xi}^\beta \partial^\gamma_\sigma R(\xi,\eta,\sigma) \right| = O \left( \frac{\left( |\xi| + |\eta| \right)^{M-|\beta|}}{|\sigma|^{M+|\alpha|}} \right).
\end{equation}
Coming back to the original symbol~(\ref{lion}), we see that
\begin{equation*}
\begin{split}
& \chi_{A_1}(\xi,\eta,\sigma) m^{III}(\xi,\eta,\sigma) m^{II}_1 (\eta,\xi) \\
& \;\;\;\;\;\;\;\;\;\;\;\;\;\;\;\;\;\;\;\;= \chi_{A_1}(\xi,\eta,\sigma) \sum_{|\alpha|=0}^{M-1} \Phi_\alpha(\sigma) m^{II}_1 (\eta,\xi) (\eta,\xi)^\alpha + \chi_{A_1}(\xi,\eta,\sigma) R(\xi,\eta,\sigma) m^{II}_1 (\eta,\xi).
\end{split}
\end{equation*}
Choosing $M$ big enough, the bounds~(\ref{boundR}) satisfied by $R$ let the symbol $\chi_{A_1}(\xi,\eta,\sigma) R(\xi,\eta,\sigma) m^{II}_1 (\eta,\xi)$ satisfy Coifman-Meyer estimates, hence the associated operator enjoys the desired bounds. Thus it suffices to treat the summands of the first term of the above right-hand side; to simplify notations a little in the following, we replace $(\xi,\eta)^\alpha$ by $(\xi-\eta)^\alpha$, 
and consider therefore symbols of the type
\begin{equation}
\label{marmotte}
\chi_{A_1}(\xi,\eta,\sigma) \Phi_\alpha(\sigma) (\xi,\eta)^\alpha  m(\eta,\xi) ,
\end{equation}
where $\Phi_\alpha $ is homogeneous of degree $-|\alpha|$, and $m$ has homogeneous bounds of degree 0. 

\bigskip

\noindent
\underline{Step 3: paraproduct decomposition of $m$}
Recall that $m$ is a symbol with homogeneous bounds of degree 0, defined at the end of the last paragraph. Its paraproduct decomposition reads
$$
B_{m}(f,g) = \sum_j B_{m}(P_j f, P_{<j-1} g) + \sum_j B_{m}(P_{< j-1} f, P_j g) + \sum_{|j-k|\leq 1} B_{m} (P_j f, P_k g) .
$$
Proceeding as in the original work of Coifman-Meyer~\cite{CoifmanMeyer}, consider the symbol of one of the elementary bilinear operators above, for instance $B_{m}(P_j \cdot, P_{<j-1} \cdot)$. Denote by  $m_j(\xi,\eta)$  this symbol, which is compactly supported (in $(\xi,\eta)$), and expand it in Fourier series 
$$
m_j(\xi,\eta) = \chi(\xi,\eta) \sum_{p,q \in \mathbb{Z}^2} a_{p,q}^j e^{i c2^{-j} (p,q)\cdot (\xi,\eta)},
$$
where we denoted $c$ for a constant, $\chi$ for a cut-off function, and $a_{p,q}^j$ for the Fourier coefficients. It is now possible to forget about the summation over $p,q$. The idea is that the fast decay of the $a^j_{p,q}$ (in $p,q$), due to the smoothness of the symbol, offsets the polynomial factors arising from the complex exponentials $e^{i c2^{-j} (p,q)}$, which correspond to translations in physical space. Les us be a bit more explicit about this: we will in the following be using maximal and square function estimates. As far as maximal functions are concerned, there holds $| \left[S_{<j} f \right](x+2^{-j} q )| \lesssim |q|^2 |M f(x)|$, which is the announced polynomial growth. As far as square functions are concerned, we rely on the boundedness of operators of the type $f \mapsto \sum_j \alpha_j [P_j f](x+q2^{-j})$ on $L^p$ spaces with $p<\infty$, if $\alpha_j \in \ell^\infty$. The bounds of these operators grow polynomially in $q$. 


These considerations reduce matters to the case where $m$ is given by  one of the three paraproduct operators
$$
(f,g) \mapsto \quad \sum_j P_j f P_{<j-1} g \quad ; \quad \sum_j P_{< j-1} f P_j g \quad ; \quad \sum_{j}  P_j f P_j g
$$
(we suppressed the index $k$ in the last summation to make notations lighter). 

\bigskip

\noindent
\underline{Step 4: derivation of the model operators}
Combining this last line with~(\ref{marmotte}), we see that the operators of interest for us become
\begin{equation}
\begin{split}
& \sum_{j,k} P_{<k-100} P_j \left(\Phi_i(D) P_k f_1 P_k f_2 \right) \nabla^i P_{<j-1} P_{<k-100} f_3 \\
& \sum_{j,k} P_{<k-100} P_{<j-1} \left(\Phi_i(D) P_k f_1 P_k f_2 \right) \nabla^i P_j P_{<k-100} f_3 \\
& \sum_{j,k} P_{<k-100} P_j \left(\Phi_i(D) P_k f_1 P_k f_2 \right) \nabla^i P_j P_{<k-100} f_3.
\end{split}
\end{equation}
where here $i = |\alpha| $ and $ \nabla^i $ stands for $ \nabla^\alpha $. 
We now make some observations which allow us to simplify the above operators:
\begin{itemize}
\item First remark that $\Phi_i(D) P_k$, $\nabla^i P_j$ and $\nabla^i P_{<j}$ can be written respectively $2^{-ik} \widetilde{P}_k$, $2^{ij} \widetilde{\widetilde{P}}_j$ and $2^{ij} \widetilde{P}_{<j}$ with obvious notations. Since the operators with tildes have very close properties to the operators without tildes, we will in the following forget about the tildes.
\item Next notice that due to the Fourier space support properties of the different terms above, it is possible to restrict the summation to $j \leq k-97$.
\item Finally, since $P_{<k-100} P_j = P_j$ and $P_{<j-1} P_{<k-100} = P_{<j-1}$ for $j \leq k-103$, it is harmless to forget about the $P_{<k-100}$ operators in the above sums.
\end{itemize}
All these remarks lead to the following simplified versions of the above operators:
\begin{subequations}
\begin{align}
\label{kingfisher1}
& \sum_{k \geq j+97} 2^{i(j-k)} P_j \left(P_k f_1 P_k f_2 \right) P_{<j-1} f_3 \\
\label{kingfisher2}
& \sum_{k \geq j+97} 2^{i(j-k)} P_{<j-1} \left(P_k f_1 P_k f_2 \right) P_j f_3 \\
\label{kingfisher3}
& \sum_{k \geq j+97} 2^{i(j-k)} P_j \left(P_k f_1 P_k f_2 \right) P_j f_3
\end{align}
\end{subequations}

\bigskip

\noindent
\underline{Step 5: the case $i=0$}
If $i=0$, observe that, due to the Fourier support properties of the Littlewood-Paley operators, the operators in~(\ref{kingfisher1}) and~(\ref{kingfisher3}) are equal respectively to
\begin{equation}
\label{whale}
\begin{split}
& \sum_j P_j \left(\sum_k P_k f_1 P_k f_2 \right) P_{<j-1} f_3 \\
& \sum_j P_j \left(\sum_k P_k f_1 P_k f_2 \right) P_j f_3
\end{split}
\end{equation}
up to a difference term which is Coifman-Meyer. But the operators in~(\ref{whale}) are simply compositions of bilinear Coifman-Meyer operators. Thus the desired bounds follow for them. 

The operator in~(\ref{kingfisher2}) can be estimated with the help of the Littlewood-Paley square and maximal function estimates (Theorem~\ref{LP}):
\begin{equation*}
\begin{split}
\left\| (\ref{kingfisher2}) \right\|_s & \lesssim \left\| \left( \sum_j \left[ P_{<j-1} \left( \sum_{k \geq j+97} P_k f_1 P_k f_2 \right) P_j f_3 \right]^2 \right)^{1/2} \right\|_s \\
& \lesssim \left\|  \left( \sum_j \left[ P_j f_3 \right]^2 \right)^{1/2} \sup_j \left| P_{<j-1} \left( \sum_{k \geq j+97} P_k f_1 P_k f_2 \right) \right| \right\|_s \\
& \lesssim \left\| S f_3 \right\|_r \left\| M\left( \sup_j \left| \sum_{k \geq j+97} P_k f_1 P_k f_2 \right| \right) \right\|_{\frac{sr}{s-r}} \\
& \lesssim \left\| f_3 \right\|_r \left\| M \left( Sf_1 Sf_2 \right) \right\|_{\frac{sr}{s-r}} \lesssim \left\| f_3 \right\|_r \left\| Sf_1 Sf_2 \right\|_{\frac{sr}{s-r}} \lesssim \left\| f_3 \right\|_r \left\| Sf_1 \right\|_p \left\| Sf_2 \right\|_q. \\
\end{split}
\end{equation*}

\bigskip

\noindent
\underline{Step 6: the case $i>0$}
If $i>0$, we see that it suffices to prove uniform estimates in $J \geq 0$ for the operators
\begin{subequations}
\begin{align}
\label{beaver1}
& \sum_{j} P_j \left(P_{j+J} f_1 P_{j+J} f_2 \right) P_{<j-1} f_3 \\
\label{beaver2} & \sum_{j} P_{<j-1} \left(P_{j+J} f_1 P_{j+J} f_2 \right) P_j f_3 \\
\label{beaver3} & \sum_{j} P_j \left(P_{j+J} f_1 P_{j+J} f_2 \right) P_j f_3
\end{align}
\end{subequations}
since the desired result follows then upon summation over $J$. We start with the case where all the Lebesgue indices are finite. The estimate relies on the Littlewood-Paley square and maximal function estimates (Theorem~\ref{LP}), and on the vector valued maximal function estimate (see Stein~\cite{Stein}, chapter II)
$$
\left\| \left( \sum_j \left[ M f_j \right]^2 \right)^{1/2} \right\|_p \lesssim \left\|\left( \sum_j f_j^2 \right)^{1/2} \right\|_p
$$ 
This gives for~(\ref{beaver1})
\begin{equation*}
\begin{split}
\left\| (\ref{beaver1}) \right\|_s & \lesssim \left\| \left( \sum_j \left[ P_j \left(P_{j+J} f_1 P_{j+J} f_2 \right) P_{<j-1} f_3 \right]^2 \right)^{1/2} \right\|_s \lesssim \left\| Mf_3 \left( \sum_j \left[ P_j \left(P_{j+J} f_1 P_{j+J} f_2 \right)\right]^2 \right)^{1/2} \right\|_s \\
& \lesssim \left\| Mf_3 \right\|_r \left\| \left(\sum_j \left[ P_j \left(P_{j+J} f_1 P_{j+J} f_2 \right)\right]^2 \right)^{1/2} \right\|_{\frac{sr}{s-r}} \lesssim \left\| f_3 \right\|_r \left\| \left(\sum_j \left[ M \left(P_{j+J} f_1 P_{j+J} f_2 \right) \right]^2 \right)^{1/2} \right\|_{\frac{sr}{s-r}} \\
& \lesssim \left\| f_3 \right\|_r \left\| \left( \sum_j \left[ P_{j+J} f_1 P_{j+J} f_2 \right]^2 \right)^{1/2} \right\|_{\frac{sr}{s-r}} \lesssim \left\| f_3 \right\|_r \left\| M f_1 \left( \sum_j \left[ P_{j+J} f_2 \right]^2 \right)^{1/2} \right\|_{\frac{sr}{s-r}}\\
&  \lesssim \left\| f_3 \right\|_r \left\| M f_1 \right\|_p \left\| S f_2 \right\|_q \lesssim  \left\| f_3 \right\|_r \left\| f_1 \right\|_p \left\| f_2 \right\|_q.
\end{split}
\end{equation*}
(\ref{beaver3}) is estimated similarly
\begin{equation*}
\begin{split}
\left\| (\ref{beaver3}) \right\|_s & \lesssim \left\| \left( \sum_j \left[ P_{j} \left(P_{j+J} f_1 P_{j+J} f_2 \right)\right]^2 \right)^{1/2} \left( \sum_j \left[ P_j f_3 \right]^2 \right)^{1/2} \right\|_s \\
& \lesssim \left\| \left( \sum_j \left[ P_{j} \left(P_{j+J} f_1 P_{j+J} f_2 \right)\right]^2 \right)^{1/2} \right\|_{\frac{sr}{s-r}} \left\| \sum_j \left[ P_j f_3 \right]^2 \right\|_r^2 \\
& \lesssim \left\| f_1 \right\|_{p} \left\| f_2 \right\|_{q} \left\| f_3 \right\|_2.
\end{split}
\end{equation*}
And finally~(\ref{beaver2}):
\begin{equation*}
\begin{split}
\left\| (\ref{beaver2}) \right\|_s & \lesssim \left\| \left( \sum_j \left[ P_{<j-1} \left(P_{j+J} f_1 P_{j+J} f_2 \right) P_j f_3 \right]^2 \right)^{1/2} \right\|_s \\
& \lesssim \left\| \sup_j \left[ P_{<j-1} \left(P_{j+J} f_1 P_{j+J} f_2 \right) \right] \right\|_{\frac{sr}{s-r}} \left\| \left( \sum_j \left[ P_j f_3 \right]^2 \right)^{1/2} \right\|_r \\
& \lesssim \left\| M \left( Mf_1 Mf_2 \right) \right\|_{\frac{sr}{s-r}} \left\| S f_3 \right\|_r \lesssim \left\|  Mf_1 Mf_2 \right\|_{\frac{sr}{s-r}} \left\| f_3 \right\|_r \\
& \lesssim \left\|  Mf_1 \right\|_{p} \left\| Mf_2 \right\|_{q} \left\| f_3 \right\|_r \lesssim \left\|  f_1 \right\|_{p} \left\| f_2 \right\|_{q} \left\| f_3 \right\|_r.
\end{split}
\end{equation*}


\bigskip

\begin{thebibliography}{num}

\bibitem{Barab} Barab, J., \emph{Nonexistence of asymptotically free solutions for a nonlinear Schr\"odinger equation.} J. Math. Phys. {\bf 25} (1984), no. 11, 3270--3273.

\bibitem{Cazenave} Cazenave, T., \emph{Semilinear Schr\"odinger equations.} Courant Lecture Notes in Mathematics, {\bf 10}. New York University, Courant Institute of Mathematical Sciences, New York; American Mathematical Society, Providence, RI, 2003

\bibitem{CW} Cazenave, T; Weissler, F, \emph{Rapidly decaying solutions of the nonlinear Schr\"odinger equation.} Comm. Math. Phys. {\bf 147} (1992),  no. 1, 75--100.

\bibitem{CoifmanMeyer} Coifman, R.; Meyer, Y., \emph{Au del\`a des op\'erateurs pseudo-diff\'erentiels.}
Ast\'erisque, {\bf 57}. Soci\'et\'e Math\'ematique de France, Paris, 1978.

\bibitem{Cohn} Cohn, S., \emph{Resonance and long time existence for the quadratic semilinear Schr\"odinger equation.} Comm. Pure Appl. Math. {\bf 45} (1992), no. 8, 973--1001.

\bibitem{Delort} Delort, J.-M., \emph{Global solutions for small nonlinear long range perturbations of two dimensional Schr\"odinger equations.}  M\'em. Soc. Math. Fr. (N.S.)  {\bf 91}  (2002).

\bibitem{GMS1} Germain, P., Masmoudi, N., Shatah, J., \emph{Global solutions for 3D quadratic Schr\"odinger equations}, accepted by International Mathematics Research Notices.

\bibitem{GMS2} Germain, P., Masmoudi, N., Shatah, J., \emph{Global solutions for the gravity water wave equations in dimension 3}, preprint.

\bibitem{GH} Ginibre, J.; Hayashi, N., \emph{Almost global existence of small solutions to quadratic nonlinear Schr\"odinger equations in three space dimensions.} Math. Z. {\bf 219} (1995), no. 1, 119--140.

\bibitem{GinibreOzawa} Ginibre, J.; Ozawa, T., \emph{Long range scattering for nonlinear Schr\"odinger and Hartree equations in space dimension $n\geq 2$.} Comm. Math. Phys. {\bf 151} (1993),  no. 3, 619--645.

\bibitem{GOV} Ginibre, J.; Ozawa, T.; Velo, G., \emph{On the existence of the wave operators for a class of nonlinear Schr\"odinger equations.} Ann. Inst. H. Poincar\'e Phys. Th\'eor. {\bf 60} (1994), no. 2, 211--239.

\bibitem{GNT1} Gustafson, S.; Nakanishi, K.; Tsai, T.-P., \emph{Scattering for the Gross-Pitaevskii equation.} Math. Res. Lett. {\bf 13} (2006),  no. 2-3, 273--285. 

\bibitem{GNT2} Gustafson, S.; Nakanishi, K.; Tsai, T.-P., \emph{Global dispersive solutions for the Gross-Pitaevskii equation in two and three dimensions.} Ann. Henri Poincar\'e {\bf 8} (2007),  no. 7, 1303--1331.

\bibitem{HMN} Hayashi, N.; Mizumachi, T.; Naumkin, P., \emph{Time decay of small solutions to quadratic nonlinear Schr\"odinger equations in 3D.} Differential Integral Equations {\bf 16} (2003), no. 2, 159--179.

\bibitem{HN1} Hayashi, N.; Naumkin, P., \emph{Asymptotics for large time of solutions to the nonlinear Schr\"odinger and Hartree equations.}  Amer. J. Math. {\bf 120}  (1998),  no. 2, 369--389.

\bibitem{HN2} Hayashi, N.; Naumkin, P., \emph{On the quadratic nonlinear Schr\"odinger equation in three space dimensions.} Internat. Math. Res. Notices {\bf 2000}, no. 3, 115--132.

\bibitem{HN3} Hayashi, N.; Miao, C.; Naumkin, P., \emph{Global existence of small solutions to the generalized derivative nonlinear Schr\"odinger equation.} Asymptot. Anal. {\bf 21} (1999), no. 2, 133--147.

\bibitem{HNST} Hayashi, N.; Naumkin, P.; Shimomura, A.; Tonegawa, S., \emph{Modified wave operators for nonlinear Schr\"odinger equations in one and two dimensions.} Electron. J. Differential Equations 2004, No. 62, 16 pp. (electronic).

\bibitem{John} John, F., \emph{Blow-up of solutions of nonlinear wave equations in three space dimensions.}  Proc. Nat. Acad. Sci. U.S.A. {\bf 76} (1979), no. 4, 1559--1560.

\bibitem{Kawahara} Kawahara, Y., \emph{Global existence and asymptotic behavior of small solutions to nonlinear Schr\"odinger equations in 3D.} Differential Integral Equations {\bf 18} (2005), no. 2, 169--194.

\bibitem{MTT} Moriyama, K.; Tonegawa, S.; Tsutsumi, Y., \emph{Wave operators for the nonlinear Schr\"odinger equation with a nonlinearity of low degree in one or two space dimensions.} Commun. Contemp. Math. {\bf 5} (2003),  no. 6, 983--996.

\bibitem{Muscalu} Muscalu, C., \emph{Paraproducts with flag singularities.} I. A case study.  Rev. Mat. Iberoam. {\bf 23} (2007),  no. 2, 705--742.

\bibitem{Nakanishi} Nakanishi, K, \emph{Asymptotically-free solutions for the short-range nonlinear Schr\"odinger equation.}  SIAM J. Math. Anal. {\bf 32} (2001), no. 6, 1265--1271.

\bibitem{Ozawa} Ozawa, T., \emph{Long range scattering for nonlinear Schr\"odinger equations in one space dimension.} Comm. Math. Phys. {\bf 139}  (1991),  no. 3, 479--493.

\bibitem{Schaeffer} Schaeffer, J., \emph{The equation $u\sb {tt}-\Delta u=\vert u\vert \sp p$ for the critical value of $p$.} Proc. Roy. Soc. Edinburgh Sect. A {\bf 101} (1985), no. 1-2, 31--44.

\bibitem{Shimomura} Shimomura, A., \emph{Nonexistence of asymptotically free solutions for quadratic nonlinear Schr\"odinger equations in two space dimensions.} Differential Integral Equations {\bf 18} (2005),  no. 3, 325--335.

\bibitem{ST2} Shimomura, A.; Tonegawa, S., \emph{Long-range scattering for nonlinear Schr\"odinger equations in one and two space dimensions.} Differential Integral Equations {\bf 17} (2004),  no. 1-2, 127--150.

\bibitem{ST} Shimomura, A.; Tsutsumi, Y., \emph{Nonexistence of scattering states for some quadratic nonlinear Schr\"odinger equations in two space dimensions.}  Differential Integral Equations {\bf 19} (2006),  no. 9, 1047--1060.

\bibitem{Stein} Stein, E., \emph{Harmonic analysis: real-variable methods, orthogonality, and oscillatory integrals.}
Princeton Mathematical Series, {\bf 43}. Monographs in Harmonic Analysis, III. Princeton University Press, Princeton, NJ, 1993.

\bibitem{str} Strauss, Walter A., \emph{Nonlinear scattering theory at low energy.}   J. Funct. Anal. 41 (1981), no. 1, 110--133.


\bibitem{Tao} Tao, T., \emph{Nonlinear dispersive equations. Local and global analysis.} CBMS Regional Conference Series in Mathematics, {\bf 106}.

\bibitem{TY} Tsutsumi, Y.; Yajima, K., \emph{The asymptotic behavior of nonlinear Schr\"odinger equations.}
Bull. Amer. Math. Soc. {\bf 11} (1984) no. 1, 186--188.

\end{thebibliography}
\end{document}